\PassOptionsToPackage{unicode}{hyperref}
\PassOptionsToPackage{naturalnames}{hyperref}
\documentclass[12pt,a4paper]{amsart}
\usepackage{amsfonts}
\usepackage{amsthm}
\usepackage{amsmath}
\usepackage{amscd}
\usepackage[latin2]{inputenc}
\usepackage{t1enc}
\usepackage[mathscr]{eucal}
\usepackage{indentfirst}
\usepackage{graphicx}
\usepackage{graphics}
\numberwithin{equation}{section}
\usepackage[margin=2.9cm]{geometry}
\usepackage{epstopdf} 
\usepackage{xcolor}

\theoremstyle{plain}
\newtheorem{Th}{Theorem}[section]
\newtheorem{lemma}[Th]{Lemma}
\newtheorem{cor}[Th]{Corollary}
\newtheorem{prop}[Th]{Proposition}

 \theoremstyle{definition}

\newtheorem{Rem}[Th]{Remark}
\newtheorem{?}[Th]{Problem}

\newcommand{\R}{\mathbb{R}}
\newcommand{\II}{{\rm{II}}}
\newcommand{\D}{\mathbb{D}}
\newcommand{\B}{\mathbb{B}}
\newcommand{\e}{{\rm {exp}}}
\newcommand{\A}{\mathcal A}
\newcommand{\F}{\mathcal F}

\newcommand{\supp}{{\rm supp}}
\begin{document}

\title[Mean First Arrival Time of Brownian Particles]{On the Mean First Arrival Time of Brownian Particles on Riemannian Manifolds}
\author[Medet Nursultanov] {M. Nursultanov}
\address {School of Mathematics and Statistics, University of Sydney}
\email{medet.nursultanov@gmail.com }

\author[Leo Tzou]{L. Tzou}
\thanks{M. Nursultanov and L. Tzou are partially supported by ARC DP190103302  and ARC DP190103451 during this work. We thank Ben Goldys for the helpful discussion.}
\address {School of Mathematics and Statistics, University of Sydney}
\email{leo.tzou@gmail.com}

\author[Justin Tzou]{J.C. Tzou}

\address {Department of Mathematics and Statistics, Macquarie University}
\email{tzou.justin@gmail.com}



\begin{abstract}
We use geometric microlocal methods to compute an asymptotic expansion of mean first arrival time for Brownian particles on Riemannian manifolds. This approach provides a robust way to treat this problem, which has thus far been limited to very special geometries. This paper can be seen as the Riemannian 3-manifold version of the planar result of \cite{ammari} and thus enable us to see the full effect of the local extrinsic boundary geometry on the mean arrival time of the Brownian particles. Our approach also connects this question to some of the recent progress on boundary rigidity and integral geometry \cite{pestovuhlmann, monard}.
\end{abstract}

\maketitle

\begin{section}{Introduction}

Let $(M,g,\partial M)$ be a compact connected orientable Riemannian manifold with non-empty smooth boundary and without loss of generality we may assume that it is an open subset of an orientable Riemannian manifold $(\tilde M, g)$ without boundary oriented by the Riemannian volume form ${\rm dvol}_g$. Let also $(X_t, \mathbb{P}_x)$ be the Brownian motion on $M$ with initial condition at $x$, that is, the stochastic process generated by the Laplace-Beltrami operator $\Delta_g$ (this article uses the convention $\Delta_g = -d^*d$ with negative spectrum, where $d$ is the exterior derivative). For any $\Gamma\subset \partial M$ open we denote by $\tau_{\Gamma}$ the first time the Brownian motion $X_t$ hits $\Gamma$, that is
	\begin{equation*}
	\tau_{\Gamma} := \inf \{ t\geq 0: X_t \in \Gamma\}.
	\end{equation*}

In the case when $\Gamma  = \Gamma_{\epsilon, a}$ is a small elliptic window of eccentricity $\sqrt{1-a^2}$ and size $\epsilon \to 0^+$ (to be made precise later), the narrow escape/mean first arrival time problem wishes to derive an asymptotic expansion as $\epsilon \to 0$ for the expected value $\mathbb{E}[\tau_{\Gamma_{\epsilon, a}} | X_0 = x]$ of the first arrival time $\tau_{\Gamma_{\epsilon, a}}$ amongst all Brownian particles starting at $x$. Another quantity of interest is the {\em average} expected value over $M$:
$$|M|^{-1} \int_M \mathbb{E}[\tau_{\Gamma_{\epsilon, a}} | X_0 = x] {\rm dvol}_g(x).$$
Here $|M|$ denotes the Riemannian volume of $M$ with respect to the metric $g$.

Many problems in cellular biology may be formulated as mean first arrival time problems; a collection of analysis methods, results, applications, and references may be found in \cite{holcman2015stochastic}. For example, cells have been modelled as simply connected two-dimensional domains with small absorbing windows on the boundary representing ion channels or target binding sites; the quantity sought is then the mean time for a diffusing ion or receptor to exit through an ion channel or reach a binding site \cite{schuss2007narrow, holcman2004escape, pillay2010asymptotic}.

There has been much progress for this problem in the setting of planar domains, and we refer the readers to \cite{holcman2004escape, pillay2010asymptotic, singer2006narrow, ammari} and references therein for a complete bibliography. An important contribution was made in the planar case by \cite{ammari} to introduce rigor into the computation of \cite{pillay2010asymptotic}. The use of layered potential in \cite{ammari} also cast this problem in the mainstream language of elliptic PDE and facilitates some of the approach we use in this article.

Few results exists for three dimensional domains in $\R^n$ or Riemannian manifolds; see \cite{cheviakov2010asymptotic,schuss2006ellipse,singer2008narrow, GomezCheviakov} and references therein. The additional difficulties introduced by higher dimension are highlighted in the introduction of \cite{ammari} and the challenges in geometry are outlined in \cite{singer2008narrow}. In the case when $M$ is a domain in $\R^3$ with Euclidean metric and $\Gamma_{\epsilon, a}$ is a single small disk absorbing window, \cite{schuss2006ellipse, singer2008narrow} gave an expansion for the average of the expected first arrival time, averaged over $M$, up to an unspecified $O(1)$ term:
\begin{eqnarray}
\label{O of 1 term}
|M|^{-1} \int_M \mathbb{E}[\tau_{\Gamma_{\epsilon, a}} | X_0 = x] {\rm dvol}_g(x) \sim \frac{|M|}{4\epsilon}\left\lbrack 1 - \frac{\epsilon}{\pi}H\log\epsilon + \mathcal{O}(\epsilon) \right\rbrack. 
\end{eqnarray}
Here, $H$ is the mean curvature of the boundary at the center of the absorbing window. The case when $\Gamma_{\epsilon,a}$ is a small elliptic window was also addressed in \cite{schuss2006ellipse, singer2008narrow}.

When $M$ is a three dimensional ball with multiple circular absorbing windows on the boundary, an expansion capturing the explicit form of the $\mathcal{O}(1)$ correction in equation \eqref{O of 1 term} in terms of the Neumann Green's function and its {\em regular part} was done in \cite{cheviakov2010asymptotic}. The method of matched asymptotic used there required the explicit computation of the Neumann Green's function, which is only possible in special geometries with high degrees of symmetry/homogeneity. In these results one does not see the full effects of local geometry. This result was also rigorously proved in \cite{ChenFriedman} but with a better estimate for the error term.

In this paper we outline an approach which allows one to derive all the main terms of $ \mathbb{E}[\tau_{\Gamma_{\epsilon, a}} | X_0 = x]$ (up to a remainder vanishing as $\epsilon \to 0$) for Riemannian manifolds of dimension three with a multiple number of small absorbing windows which are boundary geodesic balls or ellipses. We will only demonstrate this approach for one absorbing window so as to not obscure the main idea. In the case when the window is a geodesic ball our approach also adapts naturally to Riemannian manifolds of any dimension as the proof of Proposition \ref{prop G expansion} as well as the analysis for inverting a key integral equation on the ball in Section \ref{normal operator} both carry through to higher dimensions.

We discuss briefly here on how to obtain a comprehensive singularity expansion at the boundary for the Neumann Green's function on a Riemannian manifold as the Euclidean case was of interest in \cite{singer2008narrow} and \cite{cheviakov2010asymptotic}. We will define in Section 3 the Neumann Green's function $G(x,z)$ on $(M,g,\partial M)$ which satisfies
$$\Delta_g G(x,z) = - \delta_x\ ,\partial_{\nu_z} G(x,z) \mid_{z\in \partial M} =\frac{-1}{ |\partial M|},\ \ \int_{\partial M} G(x,z) {\rm dvol}_{\partial M}(z) = 0,$$
where $z\in \partial M \mapsto \nu_z$ is a outward pointing normal vector field and $|\partial M|$ is the area of the boundary.

Singer-Schuss-Holcman in \cite{singer2008narrow} highlighted the difficulty in obtaining a comprehensive singularity expansion of $G(x,z)\mid_{x,z\in\partial M, x\neq z}$ in a neighbourhood of the diagonal $\{x=z\}$ when $M$ is a bounded domain in $\R^n$, but it turns out that even when $M$ is a general Riemannian manifold this question can be treated by the standard pseudodifferential operators approach. We only carry out this calculation in three dimensions as it pertains to our application. Readers who are interested in the higher dimensional analogue can follow our treatment to carry out the (cumbersome) calculations for themselves:
\begin{prop}
\label{prop G expansion}
For $x,y\in \partial M$, set $H(x)$ to be the mean curvature of $\partial M$ at $x$, $d_h(x,y)$ the geodesic distance on the boundary given by metric $h := \iota_{\partial M}^*g$, $d_g(x,y)$ the geodesic distance given by the metric $g$, and 
$$\II_x(V) := \II_x(V,V),\ \ V\in T_x\partial M$$
the scalar second fundamental quadratic form (see pages 235 and 381 of \cite{lee} for definitions).\\
i) The map 
$$f\in C^\infty(\partial M) \mapsto \left.\left(\int_{\partial M} G(\cdot,y) f(y) d{\rm vol}_h(y)\right)\right\vert_{ \partial M}$$
is well defined and extends to a map from $H^{k}(\partial M) \to H^{k+1}(\partial M)$ for all $k\in \R$ whose Schwartz kernel we will denote by $G_{\partial M}(x,y)\in {\mathcal D}'(\partial M\times \partial M)$. Here the map $u\in H^1(M)\mapsto u\mid_{\partial M} \in H^{1/2}(\partial M)$ is the trace map. \\
ii)
There exists an open neighbourhood of the diagonal 
$${\rm Diag}:=\{(x,y)\in \partial M\times \partial M \mid x=y\} $$
 such that in this neighbourhood, the singularity structure of $G_{\partial M}(x,y)$ is given by:
\begin{eqnarray}
\label{G expansion}
\quad \quad G_{\partial M}(x,y) &=& \frac{1}{2\pi} d_g(x,y)^{-1} - \frac{1}{4\pi}  {H(x)} \log d_{h}(y,x) \\\nonumber&& + \frac{1}{16\pi} \left(\II_x\left (\frac{\exp_{x; h}^{-1} (y)}{|\exp_{x;h}^{-1} (y)|_h}\right) - \II_x\left (\frac{*\exp_{x;h}^{-1} (y)}{|\exp_{x;h}^{-1} (y)|_h}\right)\right) + R(x,y),
\end{eqnarray}
where $R(\cdot,\cdot) \in C^{0,\mu}(\partial M\times \partial M)$, for all $\mu<1$, is called {\em the regular part} of the Green's function and $*$ is the Hodge-star operator (i.e. rotation by $\pi/2$ on the surface $\partial M$).
\end{prop}

We recall the definition of the exponential map. Let $(X,g_0)$ be a geodesically complete manifold. For any $x\in X$ and $V\in T_xX$ there exists a unique geodesic $\gamma_{g_0}(t)=\gamma_{g_0}(t;V)$, defined on $[0,1]$, such that $\gamma_{g_0}(0)=x$, $\gamma_g'(0) = V$. The exponential map based at $x$ is then a map taking $T_xX \to X$ defined by 
$$\exp_{x;g_0} : V \mapsto \gamma_{g_0}(1;V).$$

Observe that when $M$ is a Euclidean ball the singular term involving the second fundamental forms vanishes due to homogeneity and therefore does not show up. This is consistent with the explicit formula derived in \cite{cheviakov2010asymptotic}.

An explicit formula for the regular part is only possible in special geometries such as the one considered in \cite{cheviakov2010asymptotic}. However, our approach in arriving at \eqref{G expansion} also provides a way to numerically compute $R(x,y)$ via a Fredholm integral equation. See Remark \ref{remark for solving for R}.

We will use the formula in Proposition \ref{prop G expansion} to derive the mean first arrival time of a Brownian particle on a Riemannian manifold with a single absorbing window which is a small geodesic ellipse. As mentioned earlier, our method extends to multiple windows but we present the single window case to simplify notations. We first state the result when the window is a geodesic {\em disk} of the boundary $\partial M$ around a fixed point since the statement is cleaner:
\begin{Th}
\label{main theorem disk}
Let $(M,g,\partial M)$ be a smooth Riemannian manifold of dimension three with boundary and let $|M|$ be its volume. \\
	i) Fix $x^*\in \partial M$ and let $\Gamma_{\epsilon}$ be a boundary geodesic ball centered at $x^*$ of geodesic radius $\epsilon >0$. For each $x\notin \Gamma_{\epsilon}$,
	$$ \mathbb{E}[\tau_{\Gamma_{\epsilon}} | X_0 = x] = F(x) + C_\epsilon -|M| G(x,x^*) + r_\epsilon (x),$$
	with $\|r_\epsilon\|_{C^k(K)} \leq C_{k,K} \epsilon$ for any integer $k$ and compact set $K\subset \overline M$ which does not contain $x^*$. The function $F$ is the unique solution to the boundary value problem
	$$\Delta_g F = -1,\ \ \partial_\nu F = -|M|/|\partial M|,\ \ \int_{\partial M} F = 0.$$
	The constant $C_\epsilon$ is, modulo an error of $O(\epsilon\log\epsilon)$, given by
	\begin{align*} \nonumber
		C_{\epsilon} = \frac{|M|}{4\epsilon} -\frac{1}{4\pi} H(x^*) |M| \log \epsilon  +  R(x^*, x^*) |M|- F(x^*) -\frac{|M|  H(x^*)}{4 \pi } \left(2\log 2 -\frac{3}{2}\right),
	\end{align*}
	where $R(x^*, x^*)$ is the evaluation at $(x,y) = (x^*, x^*)$ of the kernel $R(x,y)$ in \eqref{G expansion}. 
	
	ii) One has that the integral of $\mathbb{E}[\tau_{\Gamma_{\epsilon, a}} | X_0 = x]$ over $M$ satisfies 
	$$\int_M \mathbb{E}[\tau_{\Gamma_{\epsilon, a}} | X_0 = x] {\rm dvol}_g (x)= \int_M F(x) {\rm dvol}_g(x)+ C_\epsilon |M| - F(x^*) |M| + O(\epsilon).$$

\end{Th}

Theorem \ref{main theorem disk} does not realize the full power of Proposition \ref{prop G expansion} as it does not see the non-homogeneity of the local geometry at $x^*$ (only the mean curvature $H(x^*)$ shows up). This is due to the fact that we are looking at windows which are geodesic balls. If we replace geodesic balls with geodesic {\em ellipses}, we see that the second fundamental form term in \eqref{G expansion} contributes to a term in $\mathbb{E}[\tau_{\Gamma_{\epsilon, a}} | X_0 = x]$ which is the {\em difference} of principal curvatures. 

To this end let $E_1(x^*), E_2(x^*) \in T_{x^*}\partial M$ be the unit eigenvectors of the shape operator at $x^*$ corresponding respectively to the principal curvatures $\lambda_1(x^*),\ \lambda_2(x^*)$. For $1\geq a>0$ fixed, let $$\Gamma_{\epsilon, a} := \{\exp_{x^*;h}(\epsilon t_1 E_1(x^*) +\epsilon t_2 E_2(x^*)) \mid t_1^2 + a^{-2} t_2^2 \leq 1 \}$$ 
be a small geodesic ellipse.
\begin{Th}
	\label{main theorem}
	Let $(M,g,\partial M)$ be a smooth Riemannian manifold of dimension three with boundary. \\
	i) For each $x\in M\backslash \Gamma_{\epsilon,a}$, 
	$$ \mathbb{E}[\tau_{\Gamma_{\epsilon, a}} | X_0 = x] = F(x) + C_{\epsilon,a} -|M| G(x,x^*) + r_\epsilon (x)$$
	with $\|r_\epsilon\|_{C^k(K)} \leq C_{k,K} \epsilon$ for any integer $k$ and compact set $K\subset \overline M$ which does not contain $x^*$. The function $F$ is the unique solution to the boundary value problem
	$$\Delta_g F = -1,\ \ \partial_\nu F = -|M|/|\partial M|,\ \ \int_{\partial M} F = 0.$$
	The constant $C_{\epsilon,a}$ is given by
		\begin{align} 
	C_{\epsilon,a} =& \frac{|M| K_a}{4a\epsilon \pi^2} -\frac{1}{4\pi} H(x^*) |M| \log \epsilon  + a R(x^*, x^*) |M|- F(x^*)\\
	&\nonumber -\frac{|M|  H(x^*)}{16\pi^3 } \int_{\mathbb{D}} \frac{1}{ (1-|s'|^2)^{1/2}} \int_\D  \frac{\log\left((t_1 - s_1)^2 + a^2 (t_2-s_2)^2 \right)^{1/2}}{ (1-|t'|^2)^{1/2}} dt' ds'\\
	&\nonumber +\frac{|M|  (\lambda_1 - \lambda_2)}{64 \pi^3} \int_{\mathbb{D}} \frac{1}{ (1-|s'|^2)^{1/2}} \int_\D \frac{(t_1 - s_1)^2 - a^2 (t_2 - s_2)^2}{(t_1 - s_1)^2 + a^2 (t_2 - s_2)^2} \frac{1}{ (1-|t'|^2)^{1/2}} dt' ds'\\
	&\nonumber + O(\epsilon\log\epsilon),
	\end{align}
where $K_a = \frac{\pi}{2} \int_{0}^{2 \pi} {\left( \cos^2 \theta + \frac{\sin^2 \theta}{a^2} \right)^{-1/2}} d \theta$ and $\mathbb{D}$ is the two dimensional unit disk centered at the origin.

	ii) One has that the integral of $\mathbb{E}[\tau_{\Gamma_{\epsilon, a}} | X_0 = x]$ over $M$ satisfies
	$$\int_M \mathbb{E}[\tau_{\Gamma_{\epsilon, a}} | X_0 = x] {\rm dvol}_g (x)= \int_M F(x) {\rm dvol}_g+ C_{\epsilon,a} |M| - F(x^*) |M| + O(\epsilon).$$

\end{Th}
Note that while the dependence on the eccentricity of the ellipse is hidden in the integrals, the dependence on the difference of the principal curvatures $(\lambda_1-\lambda_2)$ is easy to see in this formula. The integral which multiplies $(\lambda_1 - \lambda_2)$ turns out to vanish when $a=1$ which makes the above result consistent with Theorem \ref{main theorem disk}.

The fact that our result is valid on general Riemannian three manifolds allows for the incorporation of spatial heterogeneity such as anisotropic diffusion. In contrast to \cite{singer2008narrow}, the fact that we are able to obtain explicitly an expression for the $O(1)$ term in \eqref{O of 1 term} is due to the fact that in Proposition \ref{prop G expansion} we have the expansion of $G_{\partial M}(x,z)$ all the way to a remainder $R(x,y)$, which is H\"older continuous at the diagonal. We also appeal to some recent advances in integral geometry \cite{sharafutdinov, monard, pestovuhlmann, pestovuhlmann2, ilmavirta} to address the comment in \cite{ammari} on the difficulty of treating this problem in higher dimensions.

The strategy and organization of this paper will be as follows. In Section 2 we will give a brief overview of pseudodifferential operators and their associated Schwartz kernels. The machinery of pseudodifferential operators serve as a bridge between the geometric and analytic objects appearing in Proposition \ref{prop G expansion} and we will compute their coordinate expression. In Section 3 we will use the tools we developed in Section 2 to prove Propostion \ref{prop G expansion}. A singularity expansion for the Green's function such as Proposition \ref{prop G expansion} is the gateway for obtaining the asymptotic expansions of Theorems \ref{main theorem disk} and \ref{main theorem}. However, there is an additional hurdle of inverting an integral transform as mentioned in \cite{ammari}. Here we make use of some recent advancements in integral geometry and geometric rigidity \cite{pestovuhlmann, ilmavirta, monard} to overcome these difficulties. This approach is described in Section 4. Finally, in Section 5 we carry out the asymptotic calculation using the tools we have developed. The appendices characterizes the expected first arrival time $\mathbb{E}[\tau_{\Gamma_{\epsilon, a}} | X_0 = x]$ as the solution of an elliptic mixed boundary value problem. This is classical in the Euclidean case (see \cite{schussbook}) but we could not find a reference for the general case of a Riemannian manifold with boundary.

\end{section}

%
%
%
%
%
%
%
%
%
%
%

%
%
%
%

\begin{section}{Overview of Pseudodifferential Operators}
\subsection{Basic Definitions}
We give some basic definitions and properties of pseudodifferential operators. For a comprehensive treatment we refer the reader to Chapt 7 of \cite{taylor2} or the book \cite{taylorpdo}. Readers who are already familiar with microlocal analysis can skip this section.

As usual, $C^\infty$ denotes the space of smooth functions. We use notation $C_c^\infty$ for compactly suported smooth functions and $D'$ for its dual. By $C^k$, we denote the space of $k$ time continuously differentiable functions. The spce of functions from $C^k$, whose $k$th derivatives are H\"{o}lder continuous with exponent $\mu\in (0,1]$, is denoted by $C^{k,\mu}$

Let $a(x,\xi)$ be a smooth function on $T^*\R^{n}$ and for all $l\in \R$ we say that $a \in S^l_1(T^*\R^n)$ (or simply $S^l_1$)  if for all multi-indices $\alpha,\beta$ there are constants $C_{\alpha,\beta}$ such that
\begin{eqnarray}
\label{symbols}
|D_\xi^\alpha D_x^\beta a(x,\xi)| \leq C_{\alpha,\beta} \langle \xi\rangle^{l - |\alpha|}
\end{eqnarray}
where $D_\xi^\alpha = (-i)^{|\alpha|}\partial_\xi^{\alpha}$, $D_x^\beta = (-i)^{|\beta|}\partial_x^{\beta}$ , and $\langle \xi\rangle := ( 1+ |\xi|^2)^{1/2}$. These are the Kohn-Nirenberg symbols. This class of symbols contain the classical symbols, denoted by $S_{cl}^{l}(T^*\R^n)$, which are defined by those $a(x,\xi)\in S^l_1(T^*\R^n)$ satisfying 
\begin{eqnarray}
\label{classical symb}
a(x,\xi) \sim \sum_{m = 0}^\infty a_{l-m}(x,\xi),
\end{eqnarray}
where each $a_{l-m}$ are homogeneous in the sense that $a_{l-m}(x, \tau \xi) = \tau^{l-m} a(x,\xi)$ for all $x\in\R^n$, $\tau>1$ and $|\xi|>1$. The expression \eqref{classical symb} means that for all $N$,
$$a(x,\xi) - \sum_{m = 0}^N a_{l-m}(x,\xi) \in S^{l-N-1}_1(T^*\R^n).$$
If $a(x,\xi)\in S^l_1$ we can define an operator $a(x,D) : C^\infty_c(\R^n) \to {\mathcal D}'(\R^n)$ by
\begin{eqnarray}
\label{integral rep}
a(x,D) u := \int_{\mathbb{R}^n} e^{i\xi\cdot x} a(x,\xi) \hat u(\xi) d\xi,
\end{eqnarray}
where $\hat u (\xi):= {\mathcal F} u := (2\pi)^{n} \int_{\mathbb{R}^n} e^{-ix\cdot\xi} u(x)dx$ is the Fourier transform. Recall that the absolutely convergent integral representation of the Fourier transform is well defined as an automorphism of the Schwartz class functions $S(\R^n)$ but extends to an automorphism of the tempered distributions $S'(\R^n)$. (See \cite{joshifriedlander} for a comprehensive guide to distribution theory and definition of these spaces).

Operators taking $C^\infty_c(\R^n) \to {\mathcal D}'(\R^n)$ which have the above representation are said to be in $\Psi^l_1(\R^n)$ and are called pseudodifferential operators. For the symbol class $S_1^l(T^*\R^n)$, Lemma 1.1 in Chapter 7.1 of \cite{taylor2} extends $a(x,D)$ to map $S'(\R^n)\to S'(\R^n)$.

 The classical pseudodifferential operators $\Psi^{l}_{cl}(\R^n)$ are defined analogously by requiring that $a(x,\xi)$ belongs to $S^l_{cl}$. Note that knowing the operator $a(x,D) \in \Psi^{l}_1(\R^n)$ we can recover $a(x,\xi)\in S^{l}_1$ by the formula
\begin{eqnarray}
\label{symbol rep}
a(x,\xi) = e_{-\xi}(x) a(x,D) e_\xi ,
\end{eqnarray}
where $e_\xi(x) := e^{i\xi\cdot x}$. Note that if $A(x,y)$ is the Schwartz kernel of the operator $a(x,D)$ then 
\begin{eqnarray}
\label{inverse transform}
a(0,\xi) = {\mathcal F}_{y}^{-1}(A(0,y))(\xi) = \int_{\mathbb{R}^n} e^{i\xi\cdot y} A(0,y) dy.
\end{eqnarray}

Let $X$ be a compact manifold without boundary. An operator $\A : C^\infty(X) \to {\mathcal D}'(X)$\footnote{We shall avoid a discussion of distributional sections of density bundles by noting that in our setting of $X$ is always prescribed with a Riemannian volume form which provides a natural trivialization of density bundles. As such all distributional sections of density bundles are identified with sections of the trivial line bundle in this way.} is said to be in $\Psi^{l}_1(X)$ if there exists coordinate covers $\{(O_j, \Phi_j)\mid \Phi_j : O_j \to U_j \subset \R^n\}$ and a partition of unity $\{\chi_j\}$ subordinate to $\{O_j\}$ such that the map
\begin{eqnarray}
\label{psido on mfld}
u\mapsto  \left (\chi_k \A \chi_j\Phi_j^* u \right)\circ \Phi_k^{-1}
\end{eqnarray}
from $C^\infty(U_j)\to {\mathcal E}'(U_k)$ belongs to $\Psi^l_1(\R^n)$. 

If $a\in C^\infty(T^*X)$ we say that it belongs to the symbol class $S^{l}_1(T^*X)$ if 
$$\chi_j\circ\Phi_j^{-1}a(\Phi_j^{-1}(\cdot), \Phi_j^*\cdot) \in S^l_1(T^*\R^n)$$
for all $j$. The classical pseudodifferential operators $\Psi^{l}_{cl}(X)$ and classical symbols $S^l_{cl}(T^*X)$ are defined analogously. These definitions depend a-priori on the choice of coordinate systems but turn out to be invariant (see Chapt 7 \cite{taylor2}).

There exists a linear isomorphism 
\begin{eqnarray}
\label{principal symbol map}
\sigma_l : \Psi^l_{cl}(X)/ \Psi^{l-1}_{cl}(X) \to S^{l}_{cl}(T^*X)/ S^{l-1}_{cl}(T^*X)
\end{eqnarray}
called the {\em principal symbol} map. For each $\A \in \Psi^l_{cl}(X)$ it can be defined at each $x\in X$ by taking a coordinate neighborhood $O$ containing $x$ and a $\chi \in C^\infty_c(O)$ which is identically $1$ near $x$ then considering the operator given in \eqref{psido on mfld} for $\chi_j = \chi_k = \chi$ and $\Phi_j = \Phi_k = \Phi$. As the resulting operator in \eqref{psido on mfld} is in $\Psi^l_{cl}(\R^n)$ with symbol $a\in S^l_{cl} (T^*\R^n)$, we may set
$$\sigma_l(\A)(x, \Phi^*\xi) := a_l(\Phi(x) , \xi)$$
for all $x\in X$ and $\xi \in T^*_{\Phi_j(x)}\R^n$. This definition depends a-priori on the choice of coordinate systems but turns out to be invariant (see Chapt 5 of \cite{taylorpdo}). In practice these computations are often done in normal coordinates centered at the point of interest $x\in X$ then computing the inverse Fourier transform as in \eqref{inverse transform}.

One important property of $\sigma_l$ which we will use is that it respects the product structure of $\Psi^l_{cl}$ and $S^l_{cl}$:
\begin{eqnarray}
\label{comp calculus}
\sigma_l(\A) \sigma_m({\mathcal B}) = \sigma_{l+m} (\A{\mathcal B})
\end{eqnarray}
for $\A\in \Psi^l_{cl}(X)$ and ${\mathcal B}\in \Psi^{m}_{cl}(X)$.

\subsection{Coordinate Calculations}

\label{diff geo}
We make some calculations for some geometric objects which will naturally appear in the singularity expansion for $G_{\partial M}$. These identities will be useful in proving Proposition \ref{prop G expansion}.

Let $(M,g,\partial M)$ be a three dimensional Riemannian manifold with non-empty smooth boundary which inherits the metric $h :=\iota_{\partial M}^*g$. Denote by $\II$ the scalar second fundamental form on the surface $\partial M$ and $H(x)$ be the mean curvature at $x\in \partial M$. Let $S\partial M$ denote the unit-sphere bundle over $\partial M$, 
$$S\partial M =\{v\in T\partial M\mid \|v\|_h =1\}.$$
For any $x\in \partial M$, let $E_1(x), E_2(x) \in S_{x} \partial M$ be principal directions (i.e. unit eigenvectors) of the induced shape operator with eigenvalues $\lambda_1(x)$ and $\lambda_2(x)$. We will drop the dependence in $x$ from our notation when there is no ambiguity.

We choose $E_1$ and $E_2$ such that $E_1^\flat\wedge E_2^\flat\wedge\nu^\flat$ is a positive multiple of the volume form $ {\rm dvol}_g$ (see p.26 of \cite{lee} for the ``musical isomorphism'' notation of $^\flat$ and $^\sharp$). Here we use $\nu$ to denote the {\em outward} pointing normal vector field so that it is consistent with most PDE literature. However, in defining $\II$ and the shape operator we will follow geometry literature (e.g.\cite{lee}) and use the {\em inward} pointing normal so that the sphere embedded in $\R^3$ would have positive mean curvature in our convention.

For two points $x,y\in \partial M$ there are two distances to consider. The first is the shortest path amongst those that stay on the boundary which we denote by $d_h(x,y)$ and the other is the distance measured by paths allowing to enter $M$, which we denote by $d_g(x,y)$. Clearly, $d_h(x,y) \geq d_g(x,y)$.

For a fixed $x_0\in \partial M$, we will denote by $B_h(\rho;x_0) \subset \partial M$ the geodesic disk of radius $\rho>0$ (with respect to the metric $h$) centered at $x_0$ and $\D_\rho$ to be the Euclidean disk in $\R^2$ of radius $\rho$ centered at the origin. In what follows $\rho$ will always be smaller than the injectivity radius of $(\partial M, h)$. Letting $t = (t_1,t_2,t_3) \in \R^3$, we will construct a coordinate system $x(t; x_0)$ by the following procedure:\\
Write $t\in \R^3$ near the origin as $t  = (t',t_3)$ for $t' = (t_1,t_2)\in \D_\rho$. Define first \footnote{for example it is obvious below that $E_1$ and $E_2$ are elements of the tangent space over $x_0$ as they are inserted into the argument of $\exp_{x_0;h}(\cdot)$.}
$$x((t',0); x_0) := {\rm {exp}}_{x_0;h} (t_1 E_1+ t_2 E_2),$$
where ${\rm{exp}}_{x_0;h} (V)$ denotes the time $1$ map of $h$-geodesics with initial point $x_0$ and initial velocity $V\in T_{x_0}\partial M$. The coordinate $t'\in \D_\rho \mapsto x((t',0); x_0)$ is then an $h$-geodesic coordinate system for a neighborhood of $x_0$ on the boundary surface $\partial M$. We can extend this to become a coordinate system for points in $M$ near $x_0$ so that $t \mapsto x(t;x_0)$ is a boundary normal coordinate system with $t_3 >0$ in $M$ as the boundary defining function. Readers wishing to know more about boundary normal coordinates can refer to \cite{leeuhlmann} for a brief recollection of the basic properties we use here and Prop 5.26 of \cite{lee} for a detailed construction.  

For convenience we will write $x(t'; x_0)$ in place of $x((t',0);x_0)$. The boundary coordinate system $ t\mapsto x(t;x_0)$ has the advantage that the metric tensor $g$ can be expressed as
\begin{eqnarray}
\label{metric bnc}
\sum_{j,k=1}^3 g_{j,k}(t) dt_j dt_k = \sum_{\alpha, \beta = 1}^2 h_{\alpha,\beta}(t',t_3) dt_\alpha dt_\beta + dt_3^2,
\end{eqnarray}
where $h_{\alpha,\beta}(t', 0) = h_{\alpha,\beta}(t')$ is the expression of the boundary metric $h$ in the $h$-geodesic coordinate system $x(t';x_0)$. Note that $(g_{j,k}(t))_{j,k=1}^3$ and $(h_{\alpha,\beta}(t',t_3))_{\alpha,\beta=1}^2$ are symmetric positive definite $3\times 3$ and $2\times 2$ matrices varying smoothly with respect to the variable $t = (t',t_3) = (t_1, t_2, t_3)$.

For $\epsilon >0$ sufficiently small we define the (rescaled) $h$-geodesic coordinate by the following map
$$ x^{\epsilon}(\cdot ; x_0) : t' = (t_1, t_2) \in \D \mapsto x(\epsilon t'; x_0) \in B_h(\epsilon;x_0), $$
where $ \mathbb{D}$ is the unit disk in $\R^2$.
We derive some coordinate expressions for some of the geometric objects we will consider later.

\begin{lemma}
\label{dist expression}
Let $h(s') =\sum_{\alpha, \beta=1}^2 h_{\alpha,\beta}(s') ds_{\alpha}ds_{\beta}$ be the pullback metric of $h$ by $s'\mapsto x(s';x_0)$ on $\D_\rho$. Denote by 
$$r = |s'-t'|_{h(s')} := \left(\sum_{\alpha, \beta = 1}^2 h_{\alpha,\beta}(s') (s_\alpha - t_\alpha)(s_\beta-t_\beta)\right)^{1/2}$$
 so that $t' = s' + r\omega$ where $\omega \in S_{s'}\D_\rho$. We have that 
$$d_h(x(s';x_0), x(t';x_0))^2 = r^2\sum\limits_{j,k = 1}^2 H_{j,k}(s',r,\omega) \omega_j \omega_k$$
for matrix $H_{j,k}(s',r, \omega)$ jointly smooth in $(s',r,\omega)$. 
 It also satisfies $H_{j,k}(s',0,\omega) = h_{j,k}(s')$ and $$\sum_{j,k}\partial_r H_{j,k} (s',0,\omega)\omega_k\omega_j= O(s').$$

\end{lemma}
\begin{proof}
Expressing $t'$ using $(s',r,\omega)$ we have that (see e.g. \cite{tibault} Lemma 4.8)
$$d_h(x(s';x_0),x(t';x_0)) = r  \left(\sum_{\alpha,\beta} H_{\alpha,\beta}(s',r,\omega) \omega_\alpha \omega_\beta\right)^{1/2}$$
with $H_{\alpha,\beta}$ symmetric, even under the map $(r,\omega)\mapsto (-r,-\omega)$, and $H_{\alpha,\beta}(s',0,\omega) = h_{\alpha,\beta}(s')$. Setting $s' = 0$ and using the fact that we are using normal coordinates we obtain
$ r^2= r^2  \left(\sum_{\alpha,\beta} H_{\alpha,\beta}(0,r,\omega) \omega_\alpha \omega_\beta\right)$. Now Taylor expanding $H_{\alpha,\beta}(0,r,\omega)$ in $r$, we see that
$$r^2 = r^2\left( 1 +  \sum_{\alpha,\beta} \partial_rH_{\alpha,\beta}(0,0,\omega) r\omega_\alpha\omega_\beta + O(r^2) \right).$$
The $r^2$ terms on left-hand and right-hand sides cancel, leaving 
$$0  =  \partial_r H_{\alpha,\beta}(0,0,\omega) r^3 \omega_\alpha \omega_\beta+ O(r^4).$$
Divide through by $r^3$ and take the limit as $r\to 0$ we see that $\sum_{\alpha,\beta}\partial_r H_{\alpha,\beta}(0,0,\omega) \omega_\alpha \omega_\beta  = 0$ for any $\omega\in S_0\D$. 
\end{proof}

\begin{lemma}
\label{d and 1/d}
We use the same notation for $\omega$ and $r$ as in Lemma \ref{dist expression}. One has that
\begin{eqnarray*} d_h(x(s';x_0),x(t';x_0))^{-1} = r^{-1} + O(s') + O(r),
\end{eqnarray*}
where $O(r)$ (respectively $O(s')$) denotes smooth functions of $(s',r,\omega) \in \D_{\rho} \times \R \times S_{s'}\D_{\rho} $ which vanish to first order as $r\to 0$ (respectively $s'\to 0$).
\end{lemma}

\begin{proof}
From Lemma \ref{dist expression} we have that
$$d_h(x(s';x_0),x(t';x_0))^{-1} = r^{-1}  \left(\sum_{\alpha,\beta} h_{\alpha,\beta}(s') \omega_\alpha \omega_\beta + r \partial_r H_{\alpha,\beta}(s',0,\omega) \omega_\alpha\omega_\beta + O(r^2)\right)^{-1/2}.$$
Using the fact that $\omega\in S_{s'}\D$ with respect to the metric given by $h_{\alpha,\beta}$ we have that 
$$d_h(x(s';x_0),x(t';x_0))^{-1} = r^{-1}  \left(\sum_{\alpha,\beta}1+ r \partial_r H_{\alpha,\beta}(s',0,\omega) \omega_\alpha\omega_\beta + O(r^2)\right)^{-1/2}.$$
For $r$ and $s'$ sufficiently small we may use Taylor's expansion to obtain the desired property. The fact that $\sum_{\alpha,\beta}\partial_r H_{\alpha,\beta}(s',0,\omega) \omega_\alpha\omega_\beta= O(s')$ is stated in Lemma \ref{dist expression}
\end{proof}
\begin{cor}
\label{rescaled boundary metric}
For $\epsilon>0$ sufficiently small we have that
\begin{eqnarray*} d_h(x^\epsilon(s';x_0),x^\epsilon(t';x_0))^{-1} = \epsilon^{-1}r^{-1} + \epsilon r^{-1} A(\epsilon, s', r,\omega)
\end{eqnarray*}
for some smooth function $A$ in the variables $(\epsilon, s',r,\omega) \in [0,\epsilon_0]\times \D\times \R \times S^1$. Here we use $r = |s'-t'|$ and $t' = s+ r\omega$.
\end{cor}

\begin{lemma}
\label{eucl norm}
Let $x(\cdot; x_0)$ be the coordinate system at the beginning of this section centered at $x_0$. For $s', t' \in \R^2$ sufficiently small, we have that
$$d_g(x(s';x_0), x(t';x_0))^2 = r^2 \left(1 +  r{\bf \tilde G}(s',\omega)   + O(r^2)\right),$$
where $r = | s'-t'|_{h(s')}$ and $t' = s' + r\omega$. Here ${\bf \tilde G}(s',\omega)$ is a smooth function of $(s',\omega)$ which vanishes at $s'  =0$. The $O(r^2)$ term is a smooth function in $(s',r,\omega)\in \D_{\rho} \times \R \times S_{s'} \D_{\rho}$ which vanishes to second order as $r\to 0$.
\end{lemma}
\begin{proof}
We begin with the identity that for any $s$ and $t$,
$$ d_g(x(s;x_0), x(t;x_0))^2 = \sum_{j,k = 1}^3 {\bf G}_{j,k}(s,t)(s_j-t_j) (s_k-t_k),$$
with ${\bf G}_{j,k}(s,s) = g_{j,k}(s)$  given by \eqref{metric bnc}. Now set $s = (s',0)$ and $t = (t',0)$ we have
$$d_g(x(s';x_0), x(t';x_0))^2 = r^2 \left( 1 + \sum_{j,k= 1}^2\partial_r {\bf G}^0_{j,k} (s',0,\omega) r \omega_j \omega_k + O(r^2)\right)$$
where ${\bf G}^0_{j,k}(s',r,\omega) := {\bf G}_{j,k}(s', s' + r\omega),$ is even in $(r,\omega) \mapsto -(r,\omega)$ and $O(r^2)$ is a smooth function of $(s',r,\omega)$ which is even and vanishes to second order as $r\to 0$. Observe that $\omega\mapsto \partial_r {\bf G}^0_{j,k}(s',0,\omega)$ is odd in $\omega$. 

We now need to argue that $\sum_{\alpha,\beta} \partial_r {\bf G}^0_{j,k}(0,0,\omega) \omega_\alpha\omega_\beta = 0$. Setting $s' =0$ in the above identity and using the fact that we are using boundary normal coordinates we have
$$ r^2 =  d_h(x(0;x_0), x(t';x_0))^2 \geq  d_g(x(0;x_0), x(t';x_0))^2 =  r^2 \left( 1 +\sum_{j,k=1}^2 \partial_r {\bf G}^0_{j,k} (0,0,\omega) r \omega_j \omega_k + O(r^2)\right).$$
Subtracting off the $r^2$ terms and dividing by $r^3$ we see that as $r\to 0$, 
$$\sum_{j,k=1}^2 \partial_r {\bf G}^0_{j,k} (0,0,\omega)  \omega_j \omega_k \leq 0$$
for all $\omega \in S^1$. We now use the fact that $\partial_r {\bf G}_{j,k}(0,0,\omega)$ is odd to see that 
$$\sum_{j,k=1}^2 \partial_r {\bf G}^0_{j,k} (0,0,\omega)  \omega_j \omega_k= 0.$$
\end{proof}
Just as how Lemma \ref{d and 1/d} and Corollary \ref{rescaled boundary metric} followed from Lemma \ref{dist expression}, we have the following:
\begin{cor}
\label{1/dg}
Using the expression $r = |s'-t'|_{h(s')}$ and $t' = s' + r \omega$, we have that
\begin{eqnarray*} d_g(x(s';x_0),x(t';x_0))^{-1} = r^{-1} + O(s') + O(r) ,
\end{eqnarray*}
where $O(r)$ (respectively $O(s')$) denotes smooth function of $(s',r,\omega) \in \D_{\rho} \times \R \times S_{s'}\D_{\rho} $  
which vanishes to first order as $r\to 0$ (respectively $s'\to 0$).
\end{cor}
\begin{cor}
\label{rescaled dist kernel}
For $\epsilon>0$ sufficiently small we have that
\begin{eqnarray*} d_g(x^\epsilon(s';x_0),x^\epsilon(t';x_0))^{-1} = \epsilon^{-1}r^{-1} + \epsilon r^{-1} A(\epsilon, s', r,\omega)
\end{eqnarray*}
for some smooth function $A$ in the variables $(\epsilon, s',r,\omega) \in [0,\epsilon_0]\times \D\times \R \times S^1$. Here we use $r = |s'-t'|$ and $t' = s+ r\omega$.
\end{cor}

\begin{lemma}
\label{normal der of distance}
In the coordinate given by $y = x(s';x_0)$,
$$\partial_{\nu_y} d_g(x_0,y) = \frac{ \lambda_1(x_0) s_1^2 + \lambda_2(x_0) s_2^2}{2|s'|} +O(|s'|^2) $$
where $O(|s'|^2)$ denotes a smooth function of the variables $(|s'|,\frac{s'}{|s'|})$ which vanishes to order 2 at the origin.
\end{lemma}
\begin{proof}
By Gauss Lemma, for $y$ near $x_0$, ${\rm grad}_y d_g(x_0,y) = \dot \gamma( d_g(x_0,y))$ where $\gamma(\cdot)$ is the unit speed geodesic in $(M,g)$ from $x_0$ to $y$. Therefore we have that
\begin{eqnarray}
\label{normal derivative of distance}
\partial_{\nu_y} d_g(x_0,y) = -\left \langle \nu_y, \frac{ \e_{y;g}^{-1}( x_0)}{|\e_{y;g}^{-1}( x_0)|_g}\right\rangle_g.
\end{eqnarray}

In  the coordinates given by $s\mapsto x(s;x_0)$ the Christoffel symbols are
\begin{eqnarray}
\label{ksymbol bnc}
\Gamma_{3,3}^3 = \Gamma_{\alpha, 3}^3 = \Gamma_{3,\alpha}^3 = 0,\ \ \Gamma_{\alpha,\beta}^3 = -\frac{1}{2}\partial_3 h_{\alpha,\beta}.
\end{eqnarray}
 
Choose $(\hat s',0)\in \R^3$ so that $x(\hat s',0;x_0) = y$. By Lemma \ref{eucl norm}, $ d_g(y, x_0)$ is a smooth function of $(|\hat s'|, \frac{\hat s'}{|\hat s'|})$ when we write $y = x(\hat s';x_0)$ in these coordinates. Let $V(y):= \frac{\e_{y;g}^{-1}(x_0)}{|\e_{y;g}^{-1}(x_0)|_g} = \sum_{j = 1}^3 V_j(\hat s') \partial_j$ be the unique unit vector over $y$ so that the $(M,g)$ unit velocity geodesic in these coordinates starting at $y$ with initial direction $V(y)$ reaches $x_0$ in time $ d_g(y, x_0)$. In the coordinates given by $s'\mapsto x(s';x_0)$, we want to argue that $V_j(\hat s')$ are smooth functions of $(|\hat s'|, \frac{\hat s'}{|\hat s'|})$. To do so, observe that in $g$-geodesic coordinates centered around $x_0$ this is of course the case. The result for any other coordinate system can then be obtained via a change of variable.

Note that the outward pointing normal is given by $-\partial_3$ since $s_3>0$ in $M$. Using \eqref{normal derivative of distance} and the expression for the metric \eqref{metric bnc} we have that in the chosen boundary normal coordinate system
\begin{eqnarray}
\label{normal derivative is V3}
\partial_{\nu_y} d_g(x_0,y) =  V_3.
\end{eqnarray}

After time $\tau$, the geodesic with initial position $(\hat s',0)$ and initial unit velocity $V$ can be written in the $s = (s',s_3)$ coordinate as
$$  s(\tau) = (\hat s',0) + \tau V + r(\tau; V)$$
for some remainder $r = (r_1, r_2, r_3)$ which has initial condition $r(0) = \dot r(0) = 0$. Taylor expanding $r(\cdot; V)$ we have that 
\begin{eqnarray}
\label{s of tau}
s(\tau) = (\hat s', 0) + \tau V + \frac{\tau^2}{2} \ddot r(0;V) + \tau^3r'(\tau; V)
\end{eqnarray}
for some $r'(\tau;V)$ depending smoothly on $\tau$ and $V$.

Due to \eqref{normal derivative is V3}, we are particularly interested in the evolution of the $r_3$ component which solves the ODE
\begin{eqnarray}
\label{r3 ode}
\ddot r_3(\tau; V) &=& -\sum_{j,k= 1}^3 \Gamma^3_{j,k} (s(\tau))(V_j + \dot r_j)(V_k +\dot r_k)\\\nonumber &=&  \frac{1}{2} \sum_{\alpha, \beta= 1}^2 \partial_3 h_{\alpha,\beta} (s(\tau))(V_\alpha + \dot r_\alpha)(V_\beta +\dot r_\beta) .
\end{eqnarray}
The last equality comes from \eqref{ksymbol bnc}.

Now set $\tau = d_g(y, x_0)$ so that $s(\tau) = 0$, we have from \eqref{s of tau} that for $\alpha = 1,2$,
$$V_\alpha =- \frac{\hat s_\alpha}{d_g(y, x_0) } +d_g(y, x_0) f_\alpha $$ 
for functions $f_\alpha$ which is smooth in $V$ and $d_g (y,x_0)$ and thus smooth functions of the $(\hat s', \hat s'/|\hat s'|)$ variable. Inserting this into \eqref{r3 ode} yields
{\small$$\frac{d_g(x,y)^2}{2} \ddot r_3(0; V) = \frac{1}{4} \sum_{\alpha, \beta} \partial_3 h_{\alpha,\beta}( 0) \hat s_\alpha \hat s_\beta + d_g(y, x)^3 f$$}%
for some function $f$ which is smooth in the variable $(|\hat s'|, \hat s'/|\hat s'|)$.
Inserting this expression into \eqref{s of tau} for $\tau = d_g(x_0,y)$ we have that
\begin{eqnarray*}V_3(\hat s') &=& - \frac{d_g(x,y)}{2} \ddot r_3(0;V)  + d_g(x_0,y)^2 r'_3(d_g(x_0,y), V(s'))\\
&=&-  \frac{1}{4} \sum_{\alpha, \beta} \frac{\partial_3 h_{\alpha,\beta}( 0)}{d_g(x_0,y)} \hat s_\alpha \hat s_\beta + d_g(y, x_0)^2 (f + r')\\
\end{eqnarray*}
Using Lemma \ref{eucl norm} we have that $d_g(x_0,y)^{-1} = |\hat s'|^{-1} + F(\frac{\hat s'}{|\hat s'|}, s')$ for some smooth function $F(\cdot, \cdot)$. We may thus write 
$$V_3(\hat s') = - \frac{1}{4} \sum_{\alpha, \beta} \frac{\partial_3 h_{\alpha,\beta}( 0)}{|\hat s'|} \hat s_\alpha \hat s_\beta + O(|\hat s'|^2) ,$$
where $O(|\hat s'|^2)$ is a smooth function of $(\frac{\hat s'}{|\hat s'|}, \hat s')$ which vanishes to order $2$ near the origin.  
Now use the fact that $-\frac{1}{2} \partial_3 h_{\alpha, \beta}(0)$ is the coordinate expression for the scalar second fundamental form with respect to the normal given by $\partial_3$ \footnote{recall that we are using the convention where $\II$ and shape operator are defined with respect to the {\em inward} pointing normal} at the point $x_0$ (see Proposition 8.17 of \cite{lee}) and our coordinate system $x(s';x_0)$ is chosen so that the shape operator is diagonalized at $x_0$. Therefore,
$$V_3(\hat s') = \frac{1}{2} \sum_{\alpha, \beta} \frac{\lambda_1 \hat s_1^2 + \lambda_2 \hat s_2^2}{|\hat s'|} + O(|\hat s'|^2).$$
In view of \eqref{normal derivative is V3} we have proven the required identity.
\end{proof}

Let $x_0\in \partial M$ and define $R_{\II}(\cdot, \cdot), R_{\II*}(\cdot, \cdot) \in L^\infty(B_h(\rho;x_0) B_h(\rho ;x_0))$ by 
\begin{eqnarray}
 R_{\II}(x,y) :=  \II_x \left(\frac{\e_ {x;h}^{-1}y}{|\e_{x;h}^{-1} y|_h}, \frac{\e _{x;h}^{-1}y}{|\e _{x,h}^{-1} y|_h}\right),\\\nonumber R_{\II*}(x,y) :=  \II_x \left(*\frac{\e_ {x;h}^{-1}y}{|\e_{x;h}^{-1} y|_h}, *\frac{\e _{x;h}^{-1}y}{|\e _{x,h}^{-1} y|_h}\right).
\end{eqnarray}
Here $*$ is the Hodge star operator associated to the metric $h$.
\begin{lemma}
\label{exp inverse vector}
Let $ x(\cdot;x_0): \D_\rho \to \partial M $ be a normal coordinate system for $(\partial M, h)$ centred around $x_0\in \partial M$ and let $h$ denote the pullback metric tensor on on $\D_\rho$ under this coordinate. Then for all $s',t'\in \D_\rho$ sufficiently close to the origin,
$$\frac{\exp^{-1}_{t'; h} (s')}{|\exp^{-1}_{t'; h} (s')|_{h(t')}} = \sum_{j = 1}^2 \omega_j \partial_j + O(t') + O(r),$$
where $\omega := \frac{s'-t'}{r}\in S^1$ with $r: = |s'-t'|$ being the Euclidean distance between $s'$ and $t'$. The $O(t')$ (resp $O(r)$) term denotes a smooth map of $(t', \omega, r ) \in \D_\rho \times S^2\times [0,\rho]$ which vanishes to order $1$ as $t'\to 0$ (resp $r\to0$).
\end{lemma}
\begin{proof}

This comes from the fact that for some matrix $H_{j,k}(s',t')$ smooth in $(s',t')$
$$d_{h}(s',t')^2 = \sum_{j=1}^2H_{j,k}(s',t')(s_j - t_j) (s_k - t_k),$$
where $H_{j,k}(t',t')  = h_{j,k}(t') =\delta_{j,k} + O(|t'|^2)$. Therefore 
$$d_{h(s',t')}(s',t') = |s'-t'| + |s'-t'| F\left(t', \frac{s'-t'}{ |s'-t'|},  |s'-t'|\right)$$
for some smooth function $F\in C^\infty(\D_\rho \times S^1\times[0,r_0])$ which is $O(t') + O(r)$.

A coordinate calculation yields that $\frac{\exp^{-1}_{t', h} (s')}{|\exp^{-1}_{t', h} (s')|_{h(t')}} = {\rm grad}_{t'} d_h(s',t')$ so 
$$\frac{\exp^{-1}_{t', h} (s')}{|\exp^{-1}_{t', h} (s')|_{h(t')}} = \sum_{j,k = 1}^2 h_{j,k}(t') \partial_{t_j} d_{h^\theta(s',t')}(s',t') \partial_k = \sum_{j} \omega_j \partial_j + O(t') + O(r).$$

\end{proof}
\begin{cor}
\label{R infty expression}
Let $x_0\in \partial M$ and $\lambda_1(x_0)$ and $\lambda_2(x_0)$ be the eigenvalues of the shape operator at $x_0$. Then,\\
i) in the $t' \mapsto x(t'; x_0)$ coordinate system prescribed at the beginning of this section,
$$ R_\II(x(t'; x_0), x(s'; x_0)) -\left( \lambda_1(x_0) \frac{(s_1 - t_1)^2}{|s' - t'|^2} + \lambda_2(x_0) \frac{(s_2 - t_2)^2}{|s' - t'|^2}\right) = O(r) + O(t')$$
$$ R_{\II*}(x(t'; x_0), x(s'; x_0)) -\left( \lambda_2(x_0) \frac{(s_1 - t_1)^2}{|s' - t'|^2} + \lambda_1(x_0) \frac{(s_2 - t_2)^2}{|s' - t'|^2}\right) = O(r) + O(t')$$

The $O(t')$ (resp $O(r)$) term denotes a smooth function of $(t', \omega, r ) \in \D_\rho \times S^2\times [0,\rho]$ which vanishes to order $1$ as $t'\to 0$ (resp $r\to0$).\\\\
ii)In the $t' \mapsto x^\epsilon(t';x_0)$ coordinate systems prescribed at the beginning of this section,
$$ R_\II(x^\epsilon(t'; x_0), x^\epsilon(s'; x_0)) -\left( \lambda_1(x_0) \frac{(s_1 - t_1)^2}{|s' - t'|^2} + \lambda_2(x_0) \frac{(s_2 - t_2)^2}{|s' - t'|^2}\right) = \epsilon R_\epsilon (t',\omega, r),$$ 
$$ R_{\II*}(x^\epsilon(t'; x_0), x^\epsilon(s'; x_0)) -\left( \lambda_2(x_0) \frac{(s_1 - t_1)^2}{|s' - t'|^2} + \lambda_1(x_0) \frac{(s_2 - t_2)^2}{|s' - t'|^2}\right) = \epsilon R_\epsilon (t',\omega, r),$$ 
where $R_\epsilon(t',\omega, r)$ is smooth with derivatives of all orders uniformly bounded in $\epsilon$.

\end{cor}
\begin{proof}
Since $x^\epsilon(t';x_0) = x(\epsilon t' ; x_0)$ we have that ii) is a consequence of i). For i), we will only prove this for $R_\II$ since the statement for $R_{\II*}$ can be obtained via a rotation.

Recall that the normal coordinate system $t'\mapsto x(t';x_0)$ is chosen so that at $x_0$ the coordinate vectors $\{\partial_{t_1},\partial_{t_2}\}$ pushes forward under $x(\cdot; x_0$ to become of eigenvectors of the shape operator. Because of this we have that the pull-back of $\II_{x}(\cdot, \cdot)$ under this coordinate system is given by $\II_{x(t';x_0)} = \sum_{j,k= 1}^2 \II_{j,k} dt_jdt_k$ where $\II_{j,k} = \lambda_j(x_0) \delta_{j,k} + O(t')$. Here $O(t')$ denotes a smooth function of $t'$ which vanishes at the origin. Using the expression derived in Lemma \ref{exp inverse vector} for the vector $\frac{\e_ {x;h}^{-1}y}{|\e_{x;h}^{-1} y|_h}$ in the coordinate given by $x = x(t';x_0)$ and $y = x(s';x_0)$,  we have the desired expression for $R_{\II}(x(t';x_0), x(s';x_0))$.
\end{proof}

\subsection{Operator Estimates}
In this section we derive Sobolev estimates for some integral kernels we will encounter when obtaining the asymptotic expansions of Theorems \ref{main theorem disk} and \ref{main theorem}. As these depend on a parameter $\epsilon>0$ and do not immediately fit into the framework of semiclassical calculus, we need to keep track of the bounds by hand.

It is useful to take the Fourier transform with respect to only some variables. Let $u(s',t')$ be a family of tempered distributions in $t'\in \R^2$ depending smoothly on the parameter $s'\in \R^2$. That is, it is the distribution defined by $\phi\mapsto \int_{\R^2} u(s',t') \phi(t') dt'$ for all $\phi\in S(\R^2)$. We denote by $\F_{t'} (u(s', t'))(\xi')$ to be the Fourier transform with respect to the $t'$ variable only.

\begin{lemma}
\label{differentiate the symbol}
Let $A(s', \omega)$ be a smooth function on $(s',\omega)\in \R^2\times S^1$. For $j \geq 0$ and $\xi' \neq 0$ we have that for any multi-index $\alpha$,
$$ D^\alpha_{s'} \F_{t'} \left( A\left( s', \frac{t'}{|t'|}\right)|t'|^{j-1}\right)\left(\xi'\right) = \F_{t'} \left( D^\alpha_{s'} A\left( s', \frac{t'}{|t'|}\right)|t'|^{j-1}\right)\left(\xi'\right) .$$
Furthermore, $ \F_{t'} \left( A\left( s', \frac{t'}{|t'|}\right)|t'|^{j-1}\right)\left(\xi'\right) $ is jointly smooth in $(s',\xi')\in \R^2 \times \R^2\backslash \{0\}$.

\end{lemma}
\begin{proof}
 Let $\chi \in C^{\infty}_c(\R^2)$ be identically $1$ near the origin. We can write for any positive integer $k$ and $\xi'\neq 0$,
\begin{eqnarray*}
 \F_{t'} \left( A\left(s', \frac{t'}{|t'|}\right)|t'|^{j-1}\right)\left(\xi'\right) &=& \F_{t'} \left( A\left( s', \frac{t'}{|t'|}\right)|t'|^{j-1}\chi(t')\right)\left(\xi'\right)  \\
&+& |\xi'|^{-2k}  \F_{t'} \left(\Delta_{t'}^k\left( A\left( s', \frac{t'}{|t'|}\right)|t'|^{j-1}(1-\chi(t'))\right)\right)\left(\xi'\right)  .\end{eqnarray*}
The first integral is absolutely convergent by the $\chi(t')$ cut-off. The second Fourier transform is also absolutely convergent owing to the fact that the integrand is smooth and $\Delta_{t'}^k$ makes the integrand decay quickly as $t'\to \infty$ provided $k$ is chosen large enough.
\end{proof}
\begin{lemma}
\label{uniform operator bound}
Let $A_\epsilon(s',r,\omega)$ be a family of $C^\infty_c( \R^2\times \R\times S^1)$ whose support is uniformly bounded in $\epsilon \in [0,\epsilon_0]$ and whose derivatives are also uniformly bounded in $\epsilon$.  Then for all $l\geq 0$,
$$\A_\epsilon f := \int_{\R^2} A_\epsilon(s',r,\omega) r^l f(s'+ r\omega) dr d\omega$$
is a map bounded uniformly in $\epsilon$ from $H^{m}(\R^2) \to H^{m+1+l}(\R^2)$.
\end{lemma}
\begin{proof}
We prove the estimates only for Schwartz functions $f \in S(\R^2)$. We first expand 
\begin{eqnarray}
\label{expand A}
A_\epsilon =\sum_{j = 1}^N \partial_r^j A_\epsilon (s',0,\omega) r^j + r^{N+1} R_\epsilon (s',r,\omega).
\end{eqnarray}
All terms are smooth in its variables with derivatives uniformly bounded in $\epsilon$. Estimating the remainder term in \eqref{expand A} is easy:
$$\A_R f := \int_{S^1}\int_0^\infty r^{N+1+l} R_\epsilon (s',r,\omega) f(s'+ r\omega) = \int_{\R^2} |s'-t'|^{N+l} R_\epsilon\left(s',|s'-t'|, \frac{s'-t'}{|s'-t'|}\right) f(t') dt'.$$ 
For any positive integer $m$ we write $f = \langle D\rangle^{2m} \langle D\rangle^{-2m}f$. Choose $N>>m$ so that there is sufficient smoothness in the integral kernel to integrate by parts the formula
$$\A_R f = \int_{\R^2} |s'-t'|^{N+l} R_\epsilon \left(s',|s'-t'|, \frac{s'-t'}{|s'-t'|}\right)  \langle D_{t'}\rangle^{2m} \langle D_{t'}\rangle^{-2m} f(t') dt'. $$
We see from this that for a fixed positive integer $m$ we may choose $N$ large enough so that $\A_R : H^{-2m}(\R^2) \to H^{2m}(\R^2)$ is uniformly bounded in $\epsilon$.

For the integral involving the main term of \eqref{expand A}, we write
$$\A_j f := \int_{S^1} \int_0^\infty \partial_r^j A_\epsilon (s',0,\omega) r^{j+l} f(s' + r\omega).$$
Let $\chi\in C^\infty_c(\R^2)$ be $1$ near the origin and write 
\begin{eqnarray}
\label{Aj split}
\A_j f = \A_j \chi(D) f + \A_j (1-\chi(D)) f.
\end{eqnarray}
To see the mapping property of the first term of \eqref{Aj split} we write it out in Cartesian coordinates
$$ \A_j \chi(D) f  = \int_{\R^2} \partial_r^j A_\epsilon\left (s',0, \frac{t'}{|t'|}\right) |t'|^{j-1+l} (\check\chi*f)(s'-t') dt'.$$
Since $\partial_r^j A_\epsilon\left (s',0, \frac{t'}{|t'|}\right)$ is smooth in $s'$ with derivatives bounded uniformly in $\epsilon$ and $\chi(D)$ is smoothing, we have that $\A_j\chi(D) : H^{-m}(\R^2) \to H^{m}(\R^2)$ for any positive integer $m$ with bound uniform in $\epsilon$. 

The second term of \eqref{Aj split} is a pseudodifferential operator with full symbol 
$$a_j(s',\xi';\epsilon) : = (1-\chi(\xi')) \F_{t'} \left( \partial_r^j A_\epsilon\left (s',0, \frac{t'}{|t'|}\right) |t'|^{j-1+l}\right)$$ 
and away from $\xi = 0$ we can deduce from Lemma \ref{differentiate the symbol} that 
$$ \F_{t'} \left( \partial_r^j A_\epsilon\left (s',0, \frac{t'}{|t'|}\right) |t'|^{j-1+l}\right) = |\xi'|^{-j - l-1} \tilde a_j(s',\xi'/|\xi'|; \epsilon)$$
for some $\tilde a_j(s',\omega; \epsilon) \in C_c^\infty (\R^2\times S^1)$ with derivatives uniformly bounded in $\epsilon$. The operator $\langle D\rangle^{m+l+1} \A_j (1-\chi(D)) \langle D\rangle^{-m}$ then has full symbol in $S^0$ with symbol seminorms bounded uniformly in $\epsilon$. We can now apply Calder\'on-Vailancourt Theorem to deduce that $\A_j (1-\chi(D))$ is bounded uniformly in $\epsilon$ from $H^m(\R^2) \to H^{m+l+1}(\R^2)$.
\end{proof}

\subsection{ Symbol Computations}
Compute symbol by taking the Fourier transform and multiply by $2\pi$. We compute the principal symbols of some of the main operators which we will encounter. The following list of inverse Fourier transforms will be useful for later computations and we will leave its proof to the reader:

\begin{lemma}
\label{some FT}
In $\R^2$ with $\xi  = (\xi_1,\xi_2)$ and $x = (x_1, x_2)$ one has that for $|\xi|\geq 1$,\\
i) ${\mathcal F}^{-1}(\log|x|) (\xi)= -2\pi|\xi|^{-2}$,\ ${\mathcal F}^{-1}(|x|^{-1}) (\xi)= 2\pi|\xi|^{-1}$\\
ii) ${\mathcal F}^{-1}(x_j |x|^{-1}) (\xi)= 2\pi i \xi_j |\xi|^{-3}$\\
iii) ${\mathcal F}^{-1}(x_j^2 |x|^{-3})(\xi) = 2\pi \xi_k^2 |\xi|^{-3}$, $k\neq j$\\
iv) ${\mathcal F}^{-1}(x_j^2 |x|^{-2}) (\xi)=2\pi (\xi_k^2 - \xi_j^2) |\xi|^{-4}$, $k\neq j$.
\end{lemma}
\begin{Rem} Note that we ignore the behaviour of $\mathcal{F}^{-1}(\cdot)$ near $\xi = 0$ as they are irrelevant to the principal symbol computations we are interested in.
\end{Rem}
\begin{lemma}
\label{symbol of dg inverse}
Let $A$ be a pseudodifferential operator on $\partial M$ whose singularity along the diagonal is given by $d_g(x,y)^{-1}$. Then $\sigma(A)(x,\xi) \equiv 2\pi|\xi|^{-1}_{h(x)}$.
\end{lemma}
\begin{proof}
We compute the symbol at $x_0\in\partial M$ using the normal coordinate $t'\mapsto x(t';x_0)$. By Corollary \ref{1/dg} $d(x(t';x_0),x_0)^{-1} = |t'|^{-1}+ O(|t'|)$. By Lemma \ref{some FT}, the inverse Fourier transform of the leading order singularity is $2\pi|\xi|^{-1}$ for $\xi$ large and therefore $\sigma(A)(x,\xi) \equiv 2\pi|\xi|^{-1}_{h(x)}$.
\end{proof}
\begin{lemma}
\label{symbol of normal derivative}
Let $A$ be a pseudodifferential operator on $\partial M$ whose singularity near the diagonal is given by $\partial_{\nu_y}d_g(x,y)^{-1}\mid_{x,y\in \partial M\times \partial M}$. Then $\sigma(A)(x_0,\xi) \equiv -\pi\frac{\II_{x_0}(*\xi^\sharp)}{|\xi|_{h(x_0)}^{3}}$
where $\xi\mapsto \xi^\sharp$ denotes the musical isomorphism from $T^*\partial M$ to $T\partial M$ induced by the boundary metric $h$ and $*$ is the Hodge star operator in this metric. Here we use $\II_{x}(V)$ to denote the quadratic form $\II_{x}(V,V)$ for $V\in T_x\partial M$.
\end{lemma}
\begin{proof}
Using Lemma \ref{normal der of distance} we see that in normal coordinates given by $x(s';x_0)$ the leading order singularity of 
$\partial_{\nu_y}d_g(x_0, y = x(s';x_0))$ is given by $-\frac{\lambda_1 s_1^2 + \lambda_2 s_2^2}{2|s'|^3}$. By Lemma \ref{some FT} we have that 
$${\mathcal F}^{-1}\left(-\frac{\lambda_1 s_1^2 + \lambda_2 s_2^2}{2|s'|^3}\right)= - \pi\frac{\lambda_1 \xi_2^2 + \lambda_2 \xi_1^2}{|\xi|^3} .$$
This is precisely the normal coordinate expression for $-\pi\frac{\II_{x_0}(*\xi^\sharp)}{|\xi|_{h(x_0)}^{3}}$.
\end{proof}

\begin{prop}
\label{principal term}
Let $H(x)$ denote the mean curvature of $\partial M$ at $x$, $\II_x$ the second fundamental form of $\partial M$ at $x\in \partial M$, and $\II_x(V) := \II_x(V,V)$ for $V\in T_x\partial M$. Define
{\small \begin{eqnarray}
\label{eq principal term}
\quad \quad K(x,y) := {H(x)\pi} \log d_{h}(y,x) - \frac{\pi}{4} \left(\II_x\left (\frac{\exp_x^{-1} (y)}{|\exp_x^{-1} (y)|_h}\right) - \II_x\left (\frac{*\exp_x^{-1} (y)}{|\exp_x^{-1} (y)|_h}\right)\right),\end{eqnarray}}
where $*$ is the Hodge star operator for the metric $h$. Let $A : C^\infty(\partial M) \to {\mathcal D}'(\partial M)$ be the operator defined by 
$$A : u \mapsto \int_{\partial M}K (x,y) u(y) {\rm dvol}_{h}.$$
Then $A \in \Psi_{cl}^{-2}(\partial M)$ with principal symbol $a(x,\xi)\in C^{\infty}(T^*\partial M\backslash \{0\})$ given by 
$$a(x,\xi) \equiv -\frac{2\pi^2}{|\xi|_h^4}\II_x(*\xi^\sharp) ,$$
where $\xi\mapsto \xi^\sharp \in T\partial M$ denotes the raising of index with respect to the metric $h$ on $\partial M$.
\end{prop}
\begin{proof}
To see that $A$ is a classical $\Psi$DO, we use Lemma \ref{dist expression} and Corollary \ref{R infty expression} to see that the coordinate expression for the integral kernel $K(x,y)$ satisfies the polyhomogeneous conditions of Prop 2.8 in Chapter 7 of \cite{taylor2}. Therefore $A\in \Psi^{-2}_{cl}(\partial M)$.

The principal symbol computation is done using normal coordinates. Fix $x_0\in \partial M$ and denote by 
$$t'\mapsto x(t' ; x_0) := \exp_{x_0; h}(t')$$
the normal coordinate system around $x_0$. By a rotation we can choose the coordinates so that $\partial_{t_j} x(0;x_0) \in T_{x_0}\partial M$ is an eigenvector of the shape operator at $x_0$ with eigenvalue $\lambda_j$. 

According to Lemma \ref{dist expression} and Corollary \ref{R infty expression}, in these coordinates the terms of $K(x_0,y)$ can be expressed as
\begin{eqnarray}
K(x_0, x(t';x_0)) = \frac{\lambda_1  +\lambda_2}{2}\pi \log|t'| - \frac{\pi}{4}\left(  \frac{\lambda_1 t_1^2 + \lambda_2t_2^2}{|t'|^2} - \frac{\lambda_1 t_2^2 + \lambda_2 t_1^2}{|t'|^2}\right)
\end{eqnarray}
for $t'$ close to the origin. Computing the principal symbol $a(x_0,\xi)$ amounts to taking the inverse Fourier transform of the above expression, and observe the behaviour as $|\xi|\to \infty$. Use the formula in Lemma \ref{some FT} , we obtain
$$a(x_0,\xi) \equiv -\frac{2\pi^2}{|\xi|^4} (\lambda_1 \xi_2^2 + \lambda_2 \xi_1^2) = -\frac{2\pi^2}{|\xi|_h^4}\II_x(*\xi^\sharp) .$$
The last equality holds due to the fact that we are using normal coordinates.
\end{proof}

\end{section}

\begin{section}{Proof of Proposition \ref{prop G expansion}}

In this section we use layer potential methods to pick out the singularity structure of the Neumann Green's function at the boundary. Assume without loss of generality that $M$ is an open subset of a compact Riemannian manifold $(\tilde M, g)$ without boundary. Choose $M' \subset \tilde M$ a manifold with boundary which compactly contains $M$. For all $F\in C^\infty_0(M')$, standard elliptic theory shows that there exists a unique solution $U_F\in C^\infty(M')$ to
$$\Delta_g U_F = F,\ \ U_F\mid_{\partial M'} = 0.$$
The map $F\mapsto U_F$ is a continuous linear operator from $C^\infty_0(M') \to {\mathcal D}'(M')$ and is therefore given by a Schwartz kernel $E(x,y) \in {\mathcal D}'(M'\times M')$ which we call the Green's function. Note that for any $u\in C_0^\infty(M')$, if we fix $x\in M'$ then by definition
$$  u(x) = \int_{M'} \Delta_gu (y) E(x,y) d{\rm vol}_g(y) = \langle u(\cdot), \Delta_g E(x,  \cdot)\rangle.$$
We formally write
 \begin{eqnarray}
 \label{global green}
 \Delta_g E(x,\cdot) = \delta_{x}(\cdot)\ \ {\rm on}\ \ M'
 \end{eqnarray}
Note that if $M$ is a bounded domain in $\R^3$ then $\tilde M$ can be chosen to be the flat torus and $E(x,y)$ can be chosen to be $\frac{-1}{4\pi|x-y|}$ for the appropriate constant $c\in \R$.

Using standard elliptic parametrix construction in normal coordinates we express $E(x,y)$ in the following way.

\begin{lemma}

For all $x,y\in M'$
	\begin{eqnarray} 
	\label{principal of E}
E(x,y) =-  \frac{1}{4\pi} d_g(x,y)^{-1} + \Psi_{cl}^{-4}( \tilde M).
	\end{eqnarray}
Here $\Psi_{cl}^{-4}( \tilde M)$ denotes the Schwartz kernel of an operator in $\Psi_{cl}^{-4}( \tilde M)$.
\end{lemma}

\begin{proof}

Let $P\in \Psi_{cl}^{-2}(\tilde M)$ be a parametrix for the $\Delta_g$ on the closed compact manifold $\tilde M$ without boundary meaning that
$$ \Delta_g P = I + \Psi^{-\infty}(\tilde M).$$
By ellipticity, for any $\chi_0\in C^\infty_c(M')$ we have 
$$\chi_0(x)\chi_0(y)(E(x,y) - P(x,y)) \in \Psi^{-\infty}(\tilde M).$$

Therefore it suffices to show that 
\begin{eqnarray}
\label{parametrix is the same}
P - P_{-2} \in \Psi_{cl}^{-4}(\tilde M)
\end{eqnarray} 
where $P_{-2}\in \Psi^{-2}(\tilde M)$ is defined by
$$(P_{-2} u)(x) := \int_{\tilde M} \frac{-\chi(x,y)}{4\pi} d_g(x,y)^{-1}u(y) d{\rm vol}_g(y)$$
for some smooth function $\chi(\cdot,\cdot) \in C^\infty(\tilde M \times \tilde M)$ satisfying $\chi(x,y) = \chi(y,x)$, $\chi(x,y) = 1$ if $d_g(x,y) <{\rm Inj}_{\tilde M}/4$ and 
$$\supp(\chi(\cdot,\cdot)) \subset\subset \{(x,y)\mid d_g(x,y) < {\rm Inj}_{\tilde M}/2\}.$$
Here ${\rm Inj}_{\tilde M}$ is the injectivity radius of the closed compact Riemannian manifold $(\tilde M, g)$.

By elliptic regularity \eqref{parametrix is the same} is equivalent to showing that
\begin{eqnarray}
\Delta_g P_{-2} - I \in \Psi_{cl}^{-2}(\tilde M).
\end{eqnarray}
Taking the adjoint  and use the self-adjointness of both $\Delta_g$ and $P_{-2}$, this is the same as
\begin{eqnarray}
P_{-2}\Delta_g  - I \in \Psi_{cl}^{-2}(\tilde M).
\end{eqnarray}

Using the principal symbol map defined in \eqref{principal symbol map} it amounts to showing that
\begin{eqnarray}
\label{-1 symbol vanishes}
\sigma_{-1}(P_{-2}\Delta_g  - I) = 0
\end{eqnarray}
as an element of the quotient space $S^{-1}_{cl}/S^{-2}_{cl}$. In fact, since the symbol is classical, we now choose $\sigma_{-1}(P_{-2} \Delta_g - I)$ to be the representative in the equivalence class which is positively homogeneous of degree $-1$.

For each $y_0\in \tilde M$ and covector $\eta_0 \in S^*_{y_0}\tilde M$ in the unit cosphere bundle we will show that 
\begin{eqnarray}
\label{homogeneous bound sym}
|\sigma_{-1}(P_{-2}\Delta_g  -I)(y_0, \tau \eta_0)| \leq C_{y_0,\eta_0} \tau^{-2}
\end{eqnarray}
as $\tau \to \infty$. Homogeneity would then ensure that $\sigma_{-1}(P_{-2}\Delta_g  -I)(y,  \eta) = 0$ for all $(y,\eta)\in T^*\tilde M$ which would then ensure \eqref{-1 symbol vanishes}.

To this end let $V_1,V_2,V_3\in S_{y_0} \tilde M$ be three orthonormal vectors and choose normal coordinate given by $\Phi(t) := \exp_{y_0}(\sum_{j=1}^3 t_j V_j)$ for $|t|_{\R^3} <{\rm Inj}_{\tilde M}/2$. Let $\chi_{\R^3}\in C_c^\infty(\R^3)$ take the value $1$ in an open set containing the origin but $\supp(\chi_{\R^3})\subset\subset \{t\in \R^3\mid |t|_{\R^3} <{\rm Inj}_{\tilde M}/2\}$. Similarly let $\chi_{\tilde M}\in C^\infty(\R^3)$ take the value $1$ in an open set containing $y_0$ but $\supp(\chi_{\tilde M})\subset\subset \{x\mid d_g(x,y_0) <{\rm Inj}_{\tilde M}/2\}$.

Define the pullback operators $A, B : C^\infty(\R^3) \to C^\infty(\R^3)$ by
$$ A : u\mapsto  \Phi_* \left(\chi_{\tilde M} P_{-2} \Phi^* \left(\chi_{\R^3} u\right)\right) ,\ B : u\mapsto  \Phi_* \left(\chi_{\tilde M} \Delta_g \Phi^* \left(\chi_{\R^3} u\right)\right)$$ 
where $\Phi^*$ and $\Phi_*$ are pullback by $\Phi$ and $\Phi^{-1}$ respectively.

Thanks to the invariance of the principal symbol map under symplectomorphism, we have
\begin{eqnarray}
\label{change of var sym}
\sigma_{-1}(P_{-2}\Delta_g - I)(y_0,\Phi^*\xi) = \sigma_{-1}(AB-I) (0, \xi)
\end{eqnarray}
for all $\xi\in T^*_0\R^3$. We see then that \eqref{-1 symbol vanishes} amounts to showing that $AB-I$ satisfies
\begin{eqnarray}
\label{full symbol calculation}
\sigma_{-1}(AB-I)(0,\xi) = 0.
\end{eqnarray}
Let $a(t,\xi)$ and $b(t,\xi)$ be the full symbol of $A$ and $B$ respectively. The full symbol of $A$ can be computed by the formula
$$a(t,\xi) = e_{-\xi}(t) \left( A e_{\xi}\right)(t)$$
where $e_\xi(t) := e^{-it\cdot \xi}$. 
Since we are using the normal coordinate around $y_0$, $d_g(y_0,\Phi(t)) = |t|_{\R^3}$ and 
$$\Phi_* d{\rm Vol}_g = \sqrt{|g|}dt =  dt +H_2(t) dt$$ 
where $H_2(t)$ is a smooth function vanishing to order $2$ at $ t= 0$. So
$$(Ae_{\xi})(0) = \int_{\R^3} \frac{-4\pi}{|t|}e^{-it\cdot \xi} dt + \int_{\R^3} \frac{c}{|t|}e^{-it\cdot \xi}\tilde H_2(t)dt$$
for some smooth and compactly supported function $\tilde H_2(t)$ vanishing to order $2$ at $t=0$. Computing the first term directly and treat the second term by expanding $H_2(t)$ using Taylor expansion we see that
\begin{eqnarray}
\label{a symbol}
a(0,\xi) = (Ae_{\xi})(0) = |\xi|^{-2} + S_{cl}^{-4}(T^*\R^3).
\end{eqnarray}
Since $B$ is the Laplace operator in the coordinate given by $\Phi$ we have that 
\begin{eqnarray}
\label{b symbol}
 b(t,\xi) = \sum_{j,k=1}^3 g^{j,k}(t) \xi_j\xi_k + \frac{1}{\sqrt{|g|}} \sum_{j,k=1}^3 \partial_{t_j} (\sqrt{|g|}g^{j,k})(t) \xi_k
\end{eqnarray}
Composition calculus gives that if $c(x,\xi)$ is the full symbol of the operator $AB$ then 
\begin{eqnarray}
\label{composition c}
c(0,\xi) = a(0,\xi) b(0,\xi) + -i \sum_{j=1}^3 \partial_{\xi_j} a(0,\xi) \partial_{t_j} b(t,\xi)\mid_{t=0} + S_{cl}^{-2}.
\end{eqnarray}
Substituting into \eqref{composition c} the expression we have in \eqref{a symbol}, \eqref{b symbol}, and the fact that in normal coordinates $g^{j,k}(t) = \delta^{j,k} + O(|t|^2)$ for $t$ in a neighbourhood of the origin we have that the second term in \eqref{composition c} drops out. So the full symbol of $AB - I$ at the point $(0,\xi)\in T^*\R^3$ is 
$$c(0,\xi) -1 \in S_{cl}^{-2}.$$
In light of \eqref{change of var sym}, for each fixed $\xi \in T^*_{y_0}\tilde M$,
$$|\sigma_{-1}(P_{-2}\Delta_g - I)(y_0,\tau \Phi^*\xi)| <C_{\xi} \tau^{-2}$$
as $\tau\to \infty$.
Therefore \eqref{homogeneous bound sym} is verified.

\end{proof}

For all $f\in C^\infty(\partial M)$ we define as in \cite{taylor2} the operators $S, N \in \Psi_{cl}^{-1}(\partial M)$ by the following
\begin{eqnarray}
\label{S and N}
Sf(x) := \int_{\partial M} E(x,y)f(y){\rm vol}_h,\ \ Nf(x) := 2\int_{\partial M} \partial_{\nu_y} E(x,y) f(y){\rm vol}_h
\end{eqnarray}
for $x \in \partial M$. Note that $Nf(x)$ is different from (see \cite{taylor2} Chapt 7 Sect 11)
$$\lim_{x \to \partial M, x\in M} 2\int_{\partial M}  \partial_{\nu_y} E(x,y) f(y) dy = f(x) + Nf(x).$$

Modulo lower order pseudodifferential operator, $S$ and $N$ are given by the integral kernels $d_g(x,y)^{-1}$ and $\partial_{\nu_y} d_g(x,y)^{-1}$ respectively. Indeed, using \eqref{principal of E} and equation (11.14) on page 38 of \cite{taylor2}, we see that for $(x,y)\in \partial M\times \partial M$ in a neighbourhood of the diagonal,
\begin{eqnarray}
\label{principal of S and N}
S = -\frac{1}{4\pi} d_g(x,y)^{-1} + \Psi_{cl}^{-3}(\partial M),\ \ N = -\frac{1}{2\pi} \partial_{\nu_y} d_g(x,y)^{-1} + \Psi_{cl}^{-2}(\partial M) .
\end{eqnarray}

Using \eqref{global green} we can construct the so called {\em Neumann Green's function} on $M$ via the following procedure. For each fixed $x\in M^o$ we can solve the following Neumann boundary value problem to obtain the correction term $C(x,y)$ as a function of $y\in M$

$$\Delta_g C(x,y) = 0,\ ,\ \partial_{\nu_y} C(x,y) = \partial_{\nu_y} E(x,y)-\frac{1}{|\partial M|},\ \ \int_{\partial M} C(x,y) d{\rm vol}_h = \int_{\partial M} E(x,y)d{\rm vol}_h.$$
Setting $G(x,y) = -E(x,y) + C(x,y)$ we get, for each fixed $x\in M$ the unique solution (as a distribution in $z$) $G(x,z)$ to
\begin{eqnarray}
\label{greens fnct}
\Delta_g G(x,z) = - \delta_x(z)\ ,\partial_{\nu_z} G(x,z) \mid_{z\in \partial M} =\frac{-1}{ |\partial M|},\ \ \int_{\partial M} G(x,z) d{\rm vol}_h = 0.
\end{eqnarray}
Fix for the time being $y\neq z$ in the interior of $M$ and observe that $x\mapsto G(z,x)$ is smooth in a neighbourhood of the singularity of the map $x\mapsto G(x,y)$ and vice versa. Therefore we can integrate by parts the the expression $G(z,y) = - \int_M G(z,x) \Delta_x G(y,x ) dx$ to obtain
$$ G(z,y) - G(y,z) = \int_{\partial M} G(y,x)\partial_{\nu_x} G(z,x)  - G(z,x)\partial_{\nu_x} G(y,x) {\rm vol}_h(x). $$
The boundary and orthogonality conditions in \eqref{greens fnct} ensures that the right side vanishes so we have
\begin{eqnarray}
\label{symmetry}
G(z,y) = G(y,z).
\end{eqnarray}

Let $\Lambda : H^{k}(\partial M) \to H^{k-1}(\partial M)$ denote the Dirichlet-to-Neumann map (see \cite{leeuhlmann} for definition) whose range is precisely a codimension one subspace of $H^{k-1}$ which annihilates the constant function. By the orthogonality condition in \eqref{greens fnct}, the behaviour of $G(x,y)$ is uniquely characterized by its action on the range of $\Lambda$. To this end, for $f\in C^\infty(\partial M)$, denote its harmonic extension by $u_f$. Integrating by parts the expression $0 = \int_{M} \Delta_g u_f(x) G(x,y) dx$ for $z$ in the interior of $M$ we have
\begin{eqnarray}
\label{compose with DN map}
u(y) &=&  \int_{\partial M}\partial_\nu u_f(x) G(x,y) {\rm dvol}_h(x)+ \frac{1}{|\partial M|} \int_{\partial M} f {\rm dvol}_h\\\nonumber
&=& \int_{\partial M}(\Lambda f)(x) G(x,y) {\rm dvol}_h(x)+ \frac{1}{|\partial M|} \int_{\partial M} f {\rm dvol}_h
\end{eqnarray}

Observe that any $\tilde f\in C^\infty(\partial M)$ has a unique decomposition $\tilde f = c + \Lambda f$ for some constant function $c$ and $f\in C^\infty(\partial M)$ satisfying $\int_{\partial M} f = 0$. Therefore, using the orthogonality condition of \eqref{greens fnct} and taking the trace of \eqref{compose with DN map} we see that the map
$$\tilde f \mapsto \left.\left(\int_{\partial M}\tilde f(x) G(x,y) {\rm dvol}_h(x)\right)\right\vert_{y\in \partial M}$$
is well defined and takes $C^\infty(\partial M) \to C^\infty(\partial M)$. We denote this operator by $G_{\partial M}$ and its Schwartz kernel by $G_{\partial M} (x,y)$.
Going back to \eqref{compose with DN map} we see that
$$f = G_{\partial M} \Lambda f + P f,$$
where 
\begin{eqnarray}
\label{P}
P f := |\partial M|^{-1} \int_{\partial M} f
\end{eqnarray}
is smoothing. In operator form this is
\begin{eqnarray}
\label{lambda inverse}
 I = G_{\partial M}\Lambda + P.
 \end{eqnarray}
Since $\Lambda \in \Psi^{1}_{cl}(\partial M)$ is elliptic (see \cite{leeuhlmann}) we can conclude that $G_{\partial M}\in \Psi_{cl}^{-1}(\partial M)$ which maps $H^{k}(\partial M) \to H^{k+1}(\partial M)$ for all $k\in \R$. This completes the proof of Proposition \ref{prop G expansion} part i). 
\begin{Rem}
A quick way to prove part ii) of Proposition \ref{prop G expansion} would be to observe that \eqref{lambda inverse} implies $G_{\partial M}$ is a parametrix for $\Lambda$. The symbol expansion for $\Lambda$ has already been computed in \cite{leeuhlmann} so constructing its parametrix follows from standard pseudodifferential calculus. However, we will choose instead to take the layer potential approach since it iwill be more conducive for future numerical implementations. See Remark \ref{remark for solving for R} below.
\end{Rem}
Applying on the right  the single-layered potential $S$ defined in\eqref{S and N} and using identity (11.58) of \cite{taylor2} we have
\begin{eqnarray}
\label{integral equation for green}
G_{\partial M} = -2S + G_{\partial M}N^* + 2PS.
\end{eqnarray}
Iterating this equation and using intertwining property (11.59) of \cite{taylor2} we get
\begin{eqnarray}
\label{G = -2S - 2NS}
G_{\partial M} = -2S - 2NS + \Psi_{cl}^{-3}(\partial M).
\end{eqnarray}
By \eqref{comp calculus} the principal symbol of the operator $NS$ is simply the product of the principal symbols of $S$ with the principal symbol of $N$. The leading singularities of the operators $S$ and $N$ are given in \eqref{principal of S and N} and the principal symbols of these kernels are computed in Lemmas \ref{symbol of dg inverse} and \ref{symbol of normal derivative}. 
Therefore, using Proposition \ref{principal term}, we see that modulo $\Psi_{cl}^{-3}(\partial M)$, the integral kernel of $NS$ is given by
$$\frac{1}{8\pi} {H(x)} \log d_{h}(y,x) - \frac{1}{32\pi} \left(\II_x\left (\frac{\exp_x^{-1} (y)}{|\exp_x^{-1} (y)|_h}\right) - \II_x\left (\frac{*\exp_x^{-1} (y)}{|\exp_x^{-1} (y)|_h}\right)\right)$$
when $x,y\in \partial M$ are close to each other.

Inserting this into \eqref{G = -2S - 2NS} we get that when $x,y\in \partial M$ are close to each other,
\begin{eqnarray}
\quad \quad G_{\partial M}(x,y) &=& \frac{1}{2\pi} d_g(x,y)^{-1} - \frac{1}{4\pi}  {H(x)} \log d_{h}(y,x) \\\nonumber&& + \frac{1}{16\pi} \left(\II_x\left (\frac{\exp_x^{-1} (y)}{|\exp_x^{-1} (y)|_h}\right) - \II_x\left (\frac{*\exp_x^{-1} (y)}{|\exp_x^{-1} (y)|_h}\right)\right) + R(x,y),
\end{eqnarray}
where $R(x,y)$ is the Schwartz kernel of an operator in $\Psi_{cl}^{-3}(\partial M)$ which we call the {\em regular part} of $G(x,y)$. Observe that since the principal symbol of $R$ is in $S_{cl}^{-3}(T^*\partial M)$ and $\partial M$ is dimension $2$, Sobolev embedding yields that 
\begin{eqnarray}
\label{regularity of R}
R(\cdot,\cdot) \in C^{0,\alpha}(\partial M\times \partial M) 
\end{eqnarray}
for all $\alpha <1$.
The proof of Proposition \ref{prop G expansion} is now complete.
\begin{Rem}
\label{remark for solving for R}
Note that \eqref{G expansion} peels off the "singular part" of the distribution $G_{\partial M}(x,y)$ and gives us the representation
$$G_{\partial M} (x,y)= G_{sing} (x,y)+ R(x,y)$$
with the singularity structure of $G_{sing}$ explicitly given by \eqref{G expansion}. Inserting this representation of $G_{\partial M}$ into \eqref{integral equation for green} gives the following integral equation for the regular part $R(x,y)$: 
$$ R(I - N^*)= -G_{sing} -2S + G_{sing} N^* + 2PS$$
where the operators $P$, $S$, and $N$ are given by \eqref{P} and \eqref{S and N}.

Since $N^* \in \Psi^{-1}_{cl}(\partial M)$, it is a compact operator which makes $I-N^*$ Fredholm with index zero. Therefore, numerically computing for $R(x,y)$ amounts to solving a Fredholm boundary integral equation subject to the orthogonality condition 
$$\int_{\partial M} G_{\partial M}(x,z) {\rm dvol}_h(z) = 0.$$
\end{Rem}

\end{section}
\begin{section}{Inverting the Normal Operator}
\label{normal operator}
Let $\Omega\subset \R^n$ be a bounded convex domain with smooth boundary. We will analyze the mapping properties of the operator 
\begin{eqnarray}
\label{def of L}
L : f\mapsto \int_\Omega \frac{f(s)}{|s-t|^{n-1}} ds_1\dots ds_n
\end{eqnarray}
and its inverse. Methods do exist \cite{Martin},\cite{HJU} for the explicit expression of the inverse of $L$ when $\Omega = \D$ (which is sufficient for our setting). When $\Omega$ is a two dimensional ellipse \cite{schuss2006ellipse} computed explicitly the inverse of $L$ acting on the constant function.

The purpose of Section \ref{bounds} is to provide a geometric perspective to the operator $L$ one of the advantages of which is that it provides an explicit formula for $L^{-1}(1)$ when $\Omega$ is a ball of any dimension. Our perspective is based on some of the recent progress on integral geometry (in particular \cite{pestovuhlmann}, \cite{monard}, \cite{ilmavirta}). Since the explicit formulas and estimates will be valid in all dimensions, this will provide the key ingredient in proving Theorem \ref{main theorem} in all dimensions. When $\Omega$ is not necessarily $\B$, this geometric point of view may also potentially provide ways to relate some of the quantities of interest to the geometry of $\Omega$.

Section \ref{integrals} will provide some explicit formulas for the composition of $L^{-1}$ with other operators in the case when $\Omega = \D$, the two dimensional disk. Section \ref{Explicit Formulas in 2D Ellipse} will do the same for when $\Omega$ is the two dimensional ellipse although the formulas will not be as explicit. 
\subsection{Mapping Properties of L} 
\label{bounds}
Denote by 
$$\partial_+S\Omega := \{(x,v) \in \partial \Omega \times S^{n-1} \mid v\cdot \nu(x) \leq 0\}$$ 
to be the set of inward pointing unit vectors on $\partial \Omega$. Note that this is a closed submanifold of the sphere bundle $S\Omega$ and thus inherits its smooth structure. Define the X-Ray transform $I : C^\infty(\overline\Omega) \to C^\infty( \partial_+ S\Omega)$ by 
$If(x,v) := \int_0^{\tau(x,v)} f(x + tv) dt$
where $\tau(x,v)$ is the time it takes for a ray of unit velocity $v$ starting at $x\in \overline\Omega$ to reach the boundary $\partial \Omega$. Note that because $\Omega$ is assumed to be convex, $\tau(x,v)$ is a smooth function on $\partial_+SM$. Furthermore, $I$ is injective by \cite{pestovuhlmann}.

By \cite{sharafutdinov} Theorem 4.2.1 this operator extends to an operator $I : L^2(\Omega) \to L^2_\mu(\partial_+S\Omega)$ 
where $\mu$ is the measure given by $\mu = |\nu(x) \cdot v| {\rm dvol}_{\partial\Omega} {\rm dvol}_{S^{n-1}}$. This $L^2$ space mapping property allows us to define the adjoint operator $I^*$ given by (see \cite{pestovuhlmann})
\begin{eqnarray}
\label{def of I*}
I^*\omega (x) = \int_{S^{n-1}} \omega( x + \tau(x,v) v) {\rm dvol}_{S^{n-1}}(v)
\end{eqnarray}
when acting on smooth functions $\omega$. This allows us to define a self-adjoint normal operator $I^*I :L^2(\Omega) \to L^2(\Omega)$. It turns out that by \cite{pestovuhlmann} the Schwartz kernel of $I^*I$ is precisely $2|s-t|^{-n+1}$ and therefore $I^* I = 2L$.
Let $d_\Omega(\cdot)$ be any smooth positive function on $\Omega$ which is equal to $dist(\cdot,\partial \Omega)$ near the boundary. By Theorem 2.2 and Theorem 4.4 of \cite{monard} respectively, we have that 
\begin{eqnarray}
\label{I*I bijective}
I^*I : d_\Omega^{-1/2} C^\infty (\overline \Omega) \to C^\infty (\overline\Omega)
\end{eqnarray}
is a bijection and 
\begin{eqnarray}
\label{I*I is homemo}
2L = I^*I : \{u \in H^{-1/2}(\R^n) \mid {\rm supp}(u) \subset \Omega\} \simeq H^{1/2}(\Omega)^*\to H^{1/2}(\Omega)
\end{eqnarray}
is a self-adjoint homeomorphism. Thus there exists a unique function $u_0 \in d_\Omega^{-1/2} C^\infty (\overline \Omega)$ such that $L u_0 = 1$ which is equivalent to $I^*I u_0= 2$.
To compute $u_0$, observe that if we find $u_0$ such that 
\begin{eqnarray}
\label{I = const}
Iu_0 (x,v) = \frac{2}{{\rm Vol}(S^{n-1})}
\end{eqnarray}
for all $(x,v) \in \partial_+ S\Omega$ then by \eqref{def of I*} we would have $I^* I u_0 = 2$. The solution of \eqref{I = const} is easy to compute explicitly when $\Omega = \B$. Indeed, direct computation shows that choosing 
\begin{eqnarray}
\label{def of u0}
u_0(x) =  \frac{2}{\pi{\rm Vol}(S^{n-1})\sqrt{1 - |x|^2}}
\end{eqnarray}
one satisfies \eqref{I = const}. In particular if $n=2$ (which is the case we are interested in) we have that 
\begin{eqnarray}
\label{L inverse}
L^{-1}(1) = u_0(x) = \frac{1}{\pi^2\sqrt{1-|x|^2}},\ \ {\rm in\ dimension\ 2}
\end{eqnarray}
\begin{Rem}
This process of computing solution to $I^*I u_0 = const$ by solving for $Iu_0  = const $ unfortunately only works for $\Omega = \B$. In fact, thanks to the rigidity result of \cite{ilmavirta}, we know that $Iu_0 = 1$ is solvable iff $\Omega = \B$. However, for more general domains the geometric view presented here could potentially allow one to apply the reconstruction formula for inverting $I$ \cite{pestovuhlmann2} to solve $I^*I u_0= 1$ explicitly. To do so one must first invert $I^*$ (which has a large kernel but is surjective) into the range of $I$ and it is not clear how to do this when $\Omega \neq \B$.
\end{Rem}

\subsection{Integrals Involving L inverse of 1.}
\label{integrals}
We also define $R_{\log}$ and $R_\infty$ to be operators with kernels $\log|s'-t'|$ and $\frac{(s_1-t_1)^2 - (s_2-t_2)^2}{|s'-t'|^2}$ respectively.

\begin{lemma}
\label{Rinfty and Rlog bounded}
The operators $R_\infty$ and $R_{\log}$ are bounded maps from $H^{1/2}(\D)^*$ to $H^{3/2}(\D)$.
\end{lemma}
\begin{proof}
Observe that the integral kernels of both $R_{\log}$ and $R_\infty$ extends naturally to kernels representing operators in $\Psi^{\infty}(\R^2)$ which we denote by $\tilde R_{\log}$ and $\tilde R_{\infty}$ respectively. We denote by $E:  H^{1/2}(\D)^* \to \{u \in H^{-1/2} (\R^2) \mid {\rm supp}(u) \subset \Omega\}$ to be the isomorphism obtained by the trivial extension. Let $\chi \in C^\infty_0(\R^2)$ be identically $1$ on $\D$. Then we have that 
\begin{eqnarray}
\label{extend kernel}
\left(R_{\log} u \right)\mid_{\D}=  \left( \chi \tilde R_{\log} \chi E u \right) \mid_{\D} 
\end{eqnarray}
and the same holds for $R_\infty$.

By Proposition 7.2.8 of \cite{taylor2} we have that both $\chi \tilde R_{\log} \chi$ and $\chi \tilde R_{\infty} \chi$ are pseudodifferential operators of order $-2$. Therefore by \eqref{extend kernel} both $R_{\log}$ and $R_{\infty}$ are bounded operators from $ H^{1/2}(\D)^*$ to $H^{3/2}(\D)$.

\end{proof}

The following lemma was proved in \cite[Theorem 4.2]{HJU}.
\begin{lemma}\label{int_L_inv}
	For $u\in H^{\frac{1}{2}}(\mathbb{D})$, it follows
	\begin{eqnarray}
	\left\langle {L}^{-1} u, 1\right\rangle = \frac{1}{\pi^2} \int_{\mathbb{D}} \frac{u(t')}{(1 - |t'|^2)^{\frac{1}{2}}}dt_1 dt_2.
	\end{eqnarray}
\end{lemma}
\begin{proof}

By \eqref{I*I is homemo} $L : H^{\frac{1}{2}}(\mathbb{D})^* \to H^{\frac{1}{2}}(\mathbb{D})$ is a self-adjoint homeomorphism. The result of this Lemma is a direct consequence.
\end{proof}
\begin{lemma}\label{log_function}
	Let
	\begin{eqnarray*}
		f(s') = R_{log}L^{-1}1 = \int_{\mathbb{D}} \log|t'-s'| [L^{-1}1](t') dt_1 dt_2,
	\end{eqnarray*}
	then
	\begin{eqnarray*}
		f(s') = \frac{2}{\pi} \log|s'| + \frac{2}{\pi}\left(\frac{1}{2} \log\left| (1 - |s'|^2)^{\frac{1}{2}} + 1 \right| - \frac{1}{2} \log\left| (1 - |s'|^2)^{\frac{1}{2}} - 1 \right|\right) - \frac{2}{\pi} (1 - |s'|^2)^{\frac{1}{2}}.
	\end{eqnarray*}   
	\begin{proof}
		Note that $\frac{1}{2\pi} \log |t '- s'|$ is the fundamental solution for the Laplace operator in $\mathbb{R}^2$, therefore,
		\begin{eqnarray*}
			\Delta f = 2 \pi [L^{-1}1] = \frac{2}{\pi} \frac{1}{(1 - |t'|^2)^{\frac{1}{2}}}
		\end{eqnarray*}
		Since $[{L}^{-1}1](t')$ is radially symmetric, $f(t') = \tilde{f}(r)$ where $r = |t'|$. Writing the Laplace operator in polar coordinates, we get
		{\Small\begin{eqnarray*}
				(r\tilde{f}_r)_r = \frac{2}{\pi} \frac{r}{(1 - r^2)^{\frac{1}{2}}}.
		\end{eqnarray*}}
		Integration gives
		{\Small\begin{eqnarray*}
				\tilde{f}_r (r)= \frac{2}{\pi} \left(\frac{C_1}{r} - \frac{(1 - r^2)^{\frac{1}{2}}}{r} \right)
		\end{eqnarray*}}
		and
		{\small\begin{eqnarray}\label{tilda_f}\nonumber
			\quad \tilde{f}(r) = \frac{2}{\pi} \left(C_1 \log r - \frac{1}{2} \log\left| (1 - r^2)^{\frac{1}{2}} - 1 \right| + \frac{1}{2} \log\left| (1 - r^2)^{\frac{1}{2}} + 1 \right| - (1 - r^2)^{\frac{1}{2}} + C_2\right).\\
			\end{eqnarray}}
		Let us find $C_1$ and $C_2$. Note that $\tilde{f}(r)$ does not have singularity at $r=0$, namely,
		{\begin{eqnarray*}
				\tilde{f}(0) = f(0) = \frac{1}{\pi^2} \int_{\mathbb{D}} \frac{\log|t'|}{(1 - |t'|^2)^{\frac{1}{2}}} = \frac{2}{\pi} \int_{0}^{1} \frac{r \log r}{(1 - r^2)^{\frac{1}{2}}} dr = \frac{2}{\pi} (\log 2 - 1).
		\end{eqnarray*}}
		Therefore, the identities
		{\begin{eqnarray*}
				C_1 \log r - \frac{1}{2} \log\left| (1 - r^2)^{\frac{1}{2}} - 1 \right| = \frac{1}{2} \log \left| \frac{r^{2 C_1}}{(1 - r^2)^{\frac{1}{2}} - 1} \right| = \frac{1}{2} \log \left| \frac{r^{2 C_1}}{-\frac{1}{2} r^2 + O(r^4)} \right|,
		\end{eqnarray*}}
		as $r \rightarrow 0$, implies that $C_1 = 1$. Hence, putting $r = 0$ into \eqref{tilda_f}, gives
		\begin{eqnarray*}
			\frac{2}{\pi} (\log 2 - 1) = \frac{2}{\pi} \left( \frac{1}{2} \log 2 + \frac{1}{2} \log2 - 1 + C_2\right),
		\end{eqnarray*}
		so that $C_2 = 0$.
	\end{proof}
\end{lemma}

Due to Lemmas \ref{log_function} and \ref{int_L_inv}, the following identity is a direct computation:
\begin{lemma}\label{integral of Rlog}
	The following identity holds
	\begin{eqnarray*}
		\left\langle L^{-1} R_{log} L^{-1}1, 1 \right\rangle = \frac{8}{\pi^2} \log 2 -\frac{6}{\pi^2}.
	\end{eqnarray*}
\end{lemma}

\begin{lemma}
\label{integral of Rinfty}
	The following idetity holds
	\begin{eqnarray*}
		\left\langle L^{-1} R_{\infty} L^{-1}1,1 \right\rangle = 0.
	\end{eqnarray*}
\end{lemma}
\begin{proof}
	By Lemmas \eqref{L inverse} and \ref{int_L_inv}, we know that
	\begin{eqnarray*}
		\left\langle L^{-1} R_{\infty} L^{-1}1,1 \right\rangle = \frac{1}{\pi^2} \int_{\mathbb{D}} \int_{\mathbb{D}} \frac{(s_1-t_1)^2 - (s_2-t_2)^2}{|s'-t'|^2} \frac{1}{(1-|s'|^2)^{\frac{1}{2}}} ds' \frac{1}{(1-|t'|^2)^{\frac{1}{2}}} dt'.
	\end{eqnarray*}
	Consider the following two changes of variables for the right-hand side
	\begin{eqnarray*}
			(s_1,s_2,t_1,t_2)=(r \cos\phi,r \sin\phi,\rho \cos\theta, \rho \sin\theta),\\
			(s_1,s_2,t_1,t_2)=(r \sin\phi,r \cos\phi,\rho \sin\theta, \rho \cos\theta).
	\end{eqnarray*}
\noindent The results differ by multiplying by $-1$, which means that the right-hand side is 0.
\end{proof}

\begin{subsection}{Explicit Formulas in 2 Dimensional  Ellipse}
\label{Explicit Formulas in 2D Ellipse}
We now compute the inverse of the map $f\mapsto \int_{{\mathcal E}_a} \frac{f(s')}{ |s'-t'|} ds'$ where the domain of integration is the two dimensional ellipse ${\mathcal E}_a := \{ s_1^2 + \frac{s_2^2}{a^2}  =1 \}$ instead of a ball. A change of variable leads us to consider the operator 
\begin{equation}
\label{La}
	L_a f = a \int_{\mathbb{D}} \frac{f(s')}{\left((t_1 - s_1)^2 + a^2(t_2 - s_2)^2\right)^{1/2}} ds'
\end{equation}
acting on functions of the disk $\D$. By \cite{schuss2006ellipse} we have that
	\begin{equation}
\label{La u = 1}
L_a  \left({K_a}^{-1} {(1 - |t'|^2)^{-1/2}}\right) = 1,
	\end{equation}
on $\D$ where 
	\begin{equation*}
	K_a = \frac{\pi}{2} \int_{0}^{2 \pi} \frac{1}{\left( \cos^2 \theta + \frac{\sin^2 \theta}{a^2} \right)^{1/2}} d \theta. 
	\end{equation*}
By \eqref{I*I is homemo} this is the unique solution in $H^{1/2}(\D)^*$ to $L_a u = 1$.

	Next we denote
	\begin{equation*}
	R_{\log,a} f(t') : = a \int_\D \log\left((t_1 - s_1)^2 + a^2 (t_2-s_2)^2 \right)^{1/2} f(s') ds',
	\end{equation*}
	
	\begin{equation*}
	R_{\infty,a} f(t') : = a \int_{\mathbb{D}} \frac{(t_1 - s_1)^2 - a^2 (t_2 - s_2)^2}{(t_1 - s_1)^2 + a^2 (t_2 - s_2)^2} f(s') ds'.
	\end{equation*}
For general $a$, the quantities $\langle L^{-1}_a(1), \R_{\log,a} L_a^{-1}(a)\rangle$ and $\langle L^{-1}_a(1), \R_{\infty,a} L_a^{-1}(a)\rangle$ cannot be computed as explicitly as in the case when $a =1$ in Section \ref{integrals}.
\end{subsection}

\end{section}
\begin{section}{Asymptotic Expansion of the Singularly Perturbed Problems}
\subsection{Asymptotic Expansion of Mixed Boundary Value Problems}
We are now ready to compute the asymptotic expansion for the mixed boundary value problem $u_\epsilon \in H^1(M)$,
\begin{eqnarray}
\label{mixed bvp}
\Delta_g u_\epsilon = -1,\ \ u_\epsilon \mid_{\Gamma_{\epsilon, a}} = 0,\ \ \partial_\nu u_\epsilon \mid_{\partial M\backslash \Gamma_{\epsilon, a}} = 0
\end{eqnarray}
which gives the compatibility condition
\begin{eqnarray}
\label{compatible}
\int_{\partial M} \partial_\nu u_\epsilon {\rm dvol}_h= -|M|.
\end{eqnarray}
All integrals on open subsets of $\partial M$ are with respect to the volume form given by the metric $h$.

Using Green's formula as in \cite{HolcmanSchuss2004}, also in \cite{ammari}, we can deduce that for points $x\in M^o$, $u_\epsilon(x)$ satisfies the integral equation
\begin{eqnarray}
\label{ueps}
u_\epsilon(x) = F(x) + C_{\epsilon,a} + \int_{\partial M } G_{}(x,y) \partial_\nu u_\epsilon(y){\rm dvol}_h(y),
\end{eqnarray}
where $C_{\epsilon,a} = |\partial M|^{-1} \int_{\partial M} u_\epsilon$ and $F(x) = \int_M G(x,y)$ solves the boundary value problem
\begin{eqnarray}
\label{F eq}
\Delta F = -1,\ \ \partial_\nu F = -|M|/|\partial M|,\ \ \int_{\partial M} F {\rm dvol}_h= 0.
\end{eqnarray}
By Proposition \ref{prop G expansion} we can take the trace of \eqref{ueps} to the boundary and restrict to the open subset $\Gamma_{\epsilon, a}^o\subset \partial M$. Using \eqref{mixed bvp} we see that
$$0 = F(x) + C_{\epsilon,a} + \int_{\Gamma_{\epsilon, a} } G_{\partial M}(x,y) \partial_\nu u_\epsilon(y){\rm dvol}_h(y) $$
for $x\in \Gamma_{\epsilon, a}^o$. We now replace $G_{\partial M}(x,y)$ for $x,y\in \Gamma_{\epsilon, a}^o$ with the expression in \eqref{G expansion} to obtain 
\begin{eqnarray}\label{integral eq}
\nonumber
-F(x) - C_{\epsilon,a} &=& \frac{1}{2\pi} \int_{\Gamma_{\epsilon, a}} d_g(x,y)^{-1} \partial_\nu u_\epsilon(y) {\rm dvol}_h(y)- \frac{H(x)}{4\pi} \int_{\Gamma_{\epsilon, a}}  \log d_{h}(x,y) \partial_\nu u_\epsilon(y){\rm dvol}_h(y)\\&& + \frac{1}{16\pi} \int_{\Gamma_{\epsilon, a}}\left(\II_x\left (\frac{\exp_x^{-1} (y)}{|\exp_x^{-1} (y)|_h}\right) - \II_x\left (\frac{*\exp_x^{-1} (y)}{|\exp_x^{-1} (y)|_h}\right)\right)\partial_\nu u_\epsilon(y) {\rm dvol}_h(y)\\&&+ \int_{\Gamma_{\epsilon, a}} R(x,y) \partial_\nu u_\epsilon(y){\rm dvol}_h(y).\nonumber
\end{eqnarray}

We will write this integral equation in the coordinate system given by 
\begin{eqnarray}\label{ellipse coord} \D \ni (s_1, s_2) \mapsto x^\epsilon(s_1, as_2; x^*) \in \Gamma_{\epsilon, a},
\end{eqnarray}
where $ x^\epsilon (\cdot; x^*): {\mathcal E}_a \to \Gamma_{\epsilon, a}$ is the coordinate defined in Section \ref{diff geo}. To simplify notation we will drop the $x^*$ in the notation and denote $x^\epsilon(\cdot; x^*)$ by simply $x^\epsilon(\cdot)$.

Note that in these coordinates the volume form for $\partial M$ is given by 
\begin{eqnarray}
\label{volume form in epsilon}
{\rm dvol}_{h} = a \epsilon^2(1 + \epsilon^2 Q_\epsilon(s') )ds_1 \wedge ds_2,\ s'\in \D
\end{eqnarray}
for some smooth function $Q_\epsilon (s')$ whose derivatives of all orders are bounded uniformly in $\epsilon$. We denote 
\begin{eqnarray}
\label{def of psi}
\psi_\epsilon(s') := \partial_\nu u_\epsilon (x^\epsilon(s_1, as_2)).
\end{eqnarray}
The compatibility condition \eqref{compatible} written using the expression for the volume form \eqref{volume form in epsilon} is then
\begin{eqnarray}
\label{average of psi}
\int_\D \psi_\epsilon (s') (1 + \epsilon^2 Q_\epsilon (s')) ds_1 ds_2 = -\frac{|M|}{a \epsilon^2}.
\end{eqnarray} 

Let us unwrap the right hand side of \eqref{integral eq} term by term in the coordinate given by $x^\epsilon(\cdot)$. Write out the integral of the first term using the expression of the volume form \eqref{volume form in epsilon} and the expression for $d_g(x,y)^{-1}$ in Corollary \ref{rescaled dist kernel} and taking into account that the coordinate system is scaled by a factor $a$ as in \eqref{ellipse coord} gives
\begin{equation}
\label{first term}
\int_{\Gamma_{\epsilon, a}} d_g(x,y)^{-1} \partial_\nu u_\epsilon(y){\rm dvol}_h(y) = a \epsilon \int_{\mathbb{D}} \frac{1}{\left((t_1 - s_1)^2 + a^2(t_2 - s_2)^2\right)^{1/2}} \psi_\epsilon(s') ds + \epsilon^3 \A_\epsilon \psi_\epsilon,
\end{equation}
for some operator $\A_\epsilon$ whose Schwartz kernel is given by the second term of the expansion in Corollary \ref{rescaled dist kernel}. Here due to Lemma \ref{uniform operator bound} we have that 
$$\A_\epsilon : H^{1/2}(\D; ds')^* \to H^{1/2}(\D; ds')$$
with operator norm bounded uniformly in $\epsilon$. From here on we will denote by $\A_\epsilon$ any operator which takes $H^{1/2}(\D ;ds')^* \to H^{1/2}(\D; ds')$ whose operator norm is bounded uniformly in $\epsilon$.

Doing the same thing for the second term of \eqref{integral eq} while using Lemma \ref{d and 1/d}, Lemma \ref{uniform operator bound}, and \eqref{average of psi} gives
\begin{eqnarray}
\label{second term}\nonumber
H(x) \int_{\Gamma_{\epsilon, a}} \log d_h (x,y) \partial_\nu u_\epsilon(y){\rm dvol}_h(y) &=& - H(x^*)|M| \log\epsilon+ \epsilon^2 H(x^*)R_{\log ,a}\psi_\epsilon + \epsilon^3 \A_\epsilon \psi_\epsilon\\&& + O_{H^{1/2} (\D)}(\epsilon \log \epsilon)
\end{eqnarray}
where $R_{\log,a}$ is defined at the very beginning of Section \ref{integrals}.
Here $ O_{H^{1/2}(\D)}(\epsilon \log \epsilon)$ denotes a function with $H^{1/2}(\D ;ds')$ norm vanishing to order $\epsilon \log\epsilon$. Note the volume for we use here is now the Euclidean one rather than ${\rm dvol}_h$ given by \eqref{volume form in epsilon}.

Finally, for the third term of \eqref{integral eq} we get by using the coordinate expression derived in Corollary \ref{R infty expression} and the estimate of Lemma \ref{uniform operator bound}:
\begin{eqnarray}
\label{third term}
\nonumber
 \int_{\Gamma_{\epsilon, a}}\left(\II_x\left (\frac{\exp_x^{-1} (y)}{|\exp_x^{-1} (y)|_h}\right) - \II_x\left (\frac{*\exp_x^{-1} (y)}{|\exp_x^{-1} (y)|_h}\right)\right)\partial_\nu u_\epsilon(y) &=&  \epsilon^2(\lambda_1-\lambda_2) R_{\infty,a} \psi_\epsilon \\
&&+ \epsilon^3 \A_\epsilon \psi_\epsilon,
\end{eqnarray}
where  $R_{\infty,a}$ is defined in Section \ref{normal operator}.
Inserting into \eqref{integral eq} the identities \eqref{first term}, \eqref{second term}, and \eqref{third term} we have
\begin{eqnarray}
\label{integral eq 2}
\nonumber
\frac{- F(x^*) - C_{\epsilon,a} - H(x^*)|M| (4\pi)^{-1} \log\epsilon }{\epsilon} &=&
\frac{1}{2\pi} L_a \psi_\epsilon - \epsilon \frac{H(x^*)}{4\pi} R_{\log,a}\psi_\epsilon+ \epsilon \frac{\lambda_1 - \lambda_2}{16\pi} R_{\infty,a} \psi_\epsilon\\&& + a \epsilon \int_\D R(x^\epsilon( t'), x^\epsilon (s'))\psi_\epsilon(s') + \epsilon^2  \A_\epsilon \psi_\epsilon \\
\nonumber&& + O_{ H^{1/2}(\D)}(\log\epsilon).
\end{eqnarray}
We would like to approximate the integral kernel $R(x^\epsilon( t'), x^\epsilon (s'))$ by the constant $R(x^*, x^*)$ and this is the content of 
\begin{lemma}
	\label{R is constant}
	Let $T_\epsilon : C_c^\infty(\D) \to {\mathcal D}'(\D)$ be the operator defined by the integral kernel 
	$$R(x^\epsilon (t'; x^*), x^\epsilon( s'; x^*)) - R(x^*, x^*).$$ 
	Then
	$$ \| T_\epsilon\|_{( H^{1/2}(\D) )^*\to H^{1/2}(\D)} \leq C \epsilon\log\epsilon.$$
\end{lemma}
\begin{proof}
	Set $T(t',s') := R(x(t';x^*), x(s';x^*)) - R(x^*, x^*)$ for $t'$ and $s'$ small and extend it to be a smooth compactly supported kernel otherwise. The kernel for $T_\epsilon$ is then $T(\epsilon t', \epsilon s')$. Note that
	\begin{eqnarray}
	\label{T vanishes at zero}
	T(0,0) = 0.
	\end{eqnarray}

	Observe that if $\chi\in C_c^\infty(\R^2)$ which is identically $1$ on $\D$ then the operator $T_\epsilon$ acting on distributions supported in $\D$ is given by the Schwartz kernel
	$$ \chi(s')\chi(t')T(\epsilon t', \epsilon s')$$ 
	for $t',s'\in \D$ and $\epsilon >0$ small.
	
	Observe that $T(t',s')$ is the integral kernel for an operator in $\Psi^{-3}_{cl}(\R^2)$. Applying Prop 2.8 in Chap 7 of \cite{taylor2} we can deduce that for all $k$ we may choose $N$ sufficiently large such that 
	$$T(t',s') = \sum_{l= 0}^N \left( q_l(t', t'-s') + p_l(t', t'-s') \log|t'-s'|\right) + R_k(t',s'),$$
	where for each integer $l$ and multi-index $\gamma$, $D^\gamma_{t'} q_l(t',\cdot)$ is a bounded continuous function of $t'$ with value in the space of smooth (away from the origin) homogeneous distributions of degree $l+1$, $p_l(t',\cdot)$ are homogenous polynomials of degree $l+1$ with coefficients which are smooth functions of $t'$, and for all multi-indices $\gamma$, $D^\gamma_{t'}R_k(t',\cdot)$ bounded continuous function of $t'$ with value in $C^k(\R^2)$.
	
	Using \eqref{T vanishes at zero} along with the homogenenity degree of $q_l$ and $p_l$ we see that $R_k(0,0) = 0$. Therefore, for $s',t' \in \D$ the integral kernel of $T_\epsilon$ is 
	{\small\begin{eqnarray}
		\label{expand Teps}\nonumber
		T_\epsilon(t',s') &=& \sum_{l= 0}^N \left( \epsilon^{l+1} q_l(\epsilon t', t'-s') + \epsilon^{l+1}p_l(\epsilon t', t'-s') \log \epsilon + \epsilon^{l+1}p_l(\epsilon t', t'-s')\log|t'-s'|\right)\\&+& R_k(\epsilon t',\epsilon s').
		\end{eqnarray}}
	The kernel $R_k(s',t')$ is sufficiently smooth. Therefore, by doing a Taylor expansion and using $R_k(0,0) = 0$ we see that the integral kernel $\chi(s')\chi(t') R_k(\epsilon t', \epsilon s')$ takes $H^{1/2}(\D)^* \to H^{1/2}(\D)$ with norm $\epsilon$. The worst term in the polyhomogeneous expansion part of \eqref{expand Teps} happens when $l = 0$ and this term is given by
	$$\epsilon q_0(\epsilon t',z') + p_0(\epsilon t', z') \log |z'| + \epsilon \log\epsilon p_0(\epsilon t',z'),$$
	where $z' = t'-s'$. Recall that both $q_0$ and $p_0$ are homogeneous of degree $1$ in $z'$ so writing $z' = r\omega$ then applying Lemma \ref{uniform operator bound} we have uniform estimates in $\epsilon$ for the kernels $\chi(s')\chi(t') q_0(\epsilon t',t'-s')$ and $\chi(s')\chi(t') p_0(\epsilon t', t'-s')$. For the term involving $\log|s'-t'|$, we use the fact that $p_0(t',z')$ is a linear function in $z'$ whose coefficients are smooth functions of $t'$. Therefore, if $u\in C^\infty_c(\R^2)$,
	\begin{eqnarray*}
		\int_{\R^2} \chi(t') p_0(\epsilon t', t'-s') \log|t'-s'| u(s') ds' &=&  \chi(t')\int_{\R^2}\sum_j c_j(\epsilon t') (t_j-s_j) \log|s'-t'| u(s') ds'\\
		&=& \chi(t') \sum_j c_j(\epsilon t') \int_{\R^2}a_j(\xi') \hat u(\xi') e^{it'\cdot \xi'}d\xi',
	\end{eqnarray*}
	where $a_j(\cdot )\in {\mathcal S}'(\R^2)$ for $j =1,2$ are derivatives of the distribution ${\rm PF} |\xi|^{-2}$ with respect to $\partial_{\xi_j}$. We refer the reader to (8.31) in Chapt 3 of \cite{taylor1} for the definition of the the distribution ${\rm PF} |\xi|^{-2}$.
	From this we see that the integral kernel 
	$$\chi(t') \chi(s') p_0(\epsilon t', t'-s') \log|t'-s'|$$
	also maps $H^{1/2}(\D)^* \to H^{1/2}(\D)$ with uniform bound in $\epsilon$. 
\end{proof}

Due to Lemma \ref{R is constant} we can write \eqref{integral eq 2} as
\begin{eqnarray}
\label{integral eq 3}
\nonumber
\frac{2\pi}{ \epsilon}\left(a R(x^*, x^*) |M|- F(x^*) - C_{\epsilon,a}  -  \frac{H(x^*)|M| \log \epsilon}{4\pi} \right) =& \left ( L_a - \frac{\epsilon H(x^*)}{2} R_{\log,a}+ \frac{\epsilon (\lambda_1-\lambda_2)}{8} R_{\infty, a}\right) \psi_\epsilon\\
&+ \epsilon T_\epsilon \psi_\epsilon+ + \epsilon^2 \A_\epsilon \psi_\epsilon  + O_{ H^{1/2}(\D)}(\log\epsilon).
\end{eqnarray}
We hit both sides with $L_a^{-1}$ and use \eqref{La u = 1} and \eqref{I*I is homemo} to get the identity
\begin{eqnarray}
\label{main integral eq}
&&\frac{2\pi}{ \epsilon}\left( a R(x^*, x^*) |M|- F(x^*) - C_{\epsilon,a}  -  \frac{H(x^*)|M| \log \epsilon}{4\pi} \right) \frac{1}{ K_a (1-|t'|^2)^{1/2}}=\\\nonumber
&& \left(I - \frac{\epsilon H(x^*)}{2} L_a^{-1} R_{\log,a} + \frac{\epsilon (\lambda_1-\lambda_2)}{8} L_a^{-1}R_{\infty,a} + T'_\epsilon\right)\psi_\epsilon + O_{ H^{1/2}(\D)^*} (\log\epsilon).
\end{eqnarray}
for some $T'_\epsilon :  H^{1/2}(\D)^* \to  H^{1/2}(\D)^*$ with operator norm $O(\epsilon^2\log\epsilon)$.
Use the mapping properties from Lemma \ref{Rinfty and Rlog bounded} we see that the right side can be inverted by Neumann series to deduce 
\begin{eqnarray}
\label{first psi asym}
\psi_\epsilon =   -\frac{2 \pi C_{\epsilon,a}}{\epsilon K_a (1-|t'|^2)^{1/2}} + C_{\epsilon,a} O_{ H^{1/2}(\D)^*}(1) + O_{ H^{1/2}(\D)^*}(\epsilon^{-1}\log\epsilon).
\end{eqnarray}

Insert \eqref{first psi asym} into \eqref{average of psi} we get that
\begin{eqnarray}
\label{first C asym}
C_{\epsilon,a} = \frac{|M| K_a}{4a\epsilon \pi^2} + C'_{\epsilon,a}
\end{eqnarray}
with $C'_{\epsilon,a} = O(\log\epsilon)$. Insert \eqref{first C asym} into \eqref{first psi asym}
\begin{eqnarray}
\label{second psi asym}
\psi_\epsilon = \frac{-|M|}{2 a \pi \epsilon^2 (1- |t'|^2)^{1/2}} +\psi_\epsilon' = -\frac{|M| K_a}{2 a \pi \epsilon^2}L_a^{-1}(1) + \psi'_\epsilon
\end{eqnarray}
with $ \|\psi_\epsilon'\|_{H^{1/2}(\D;ds')^*} \leq C\epsilon^{-1}\log\epsilon$. 
Insert \eqref{first C asym} and \eqref{second psi asym} into \eqref{main integral eq} we get 

{\small \begin{eqnarray}
	\label{main integral eq2}
	&&\frac{2\pi}{ \epsilon}\left( a R(x^*, x^*) |M|- F(x^*) - C'_{\epsilon,a}  -  \frac{H(x^*)|M| \log \epsilon}{4\pi} \right) \frac{1}{K_a (1-|t'|^2)^{1/2}}=\\\nonumber
	&&  \psi'_\epsilon +  \frac{|M|K_a}{2 \pi \epsilon}\left(  \frac{ H(x^*)}{2} L_a^{-1} R_{\log,a} -  \frac{(\lambda_1-\lambda_2)}{8} L_a^{-1}R_{\infty,a}\right)L_a^{-1}(1)   + O_{ H^{1/2}(\D)^*} (\log\epsilon).
	\end{eqnarray}}
Inserting the expression \eqref{second psi asym} into \eqref{average of psi} we get that
\begin{eqnarray}
\label{log limit of average}
\int_\D \psi'_\epsilon(s')ds_1 ds_2 = O(1).
\end{eqnarray}
Multiply \eqref{main integral eq2} by $\epsilon$ then integrate over $\D$ . Then \eqref{log limit of average} implies
\begin{eqnarray} 
C'_{\epsilon,a} &=& -\frac{1}{4\pi} H(x^*) |M| \log \epsilon  + a R(x^*, x^*) |M|- F(x^*)\\\nonumber
&& -\frac{|M| H(x^*) K_a^2 }{16\pi^3} \int_{\D} L_a^{-1} R_{\log,a} L_a^{-1} 1(s') ds'\\\nonumber
&& +\frac{|M|  (\lambda_1 - \lambda_2) K_a^2 }{64 \pi^3} \int_{\D} L_a^{-1} R_{\infty,a} L_a^{-1} 1(s') ds' + O(\epsilon\log\epsilon).
\end{eqnarray}

Since $L_a^{-1}$ is self-adjoint, we can express the last two integrals more explicitly:
\begin{eqnarray*}
\int_{\D} L_a^{-1} R_{\log,a} L_a^{-1} 1(s') ds' = K_a^{-2} \langle (1-|s'|^2)^{-1/2}, R_{\log,a}(1-|s'|^2)^{-1/2}\rangle,
\end{eqnarray*}
\begin{eqnarray*}
\int_{\D} L_a^{-1} R_{\infty,a} L_a^{-1} 1(s') ds' = K_a^{-2} \langle (1-|s'|^2)^{-1/2}, R_{\infty,a}(1-|s'|^2)^{-1/2}\rangle.
\end{eqnarray*}

We summarize this calculation into the following: 
\begin{prop}
	\label{expansion of psi and Ceps}
	We have that 
	$$\psi_\epsilon = \frac{-|M|}{2 a \pi \epsilon^2 (1- |t'|^2)^{1/2}} +\psi_\epsilon' $$
	with $\psi_\epsilon' =  O_{H^{1/2}(\D)^*}(\epsilon^{-1}\log\epsilon)$. Furthermore 
	\begin{align} 
\label{Cepsa}
		C_{\epsilon,a} =& \frac{|M| K_a}{4a\epsilon \pi^2} -\frac{1}{4\pi} H(x^*) |M| \log \epsilon  + a R(x^*, x^*) |M|- F(x^*)\\
		&\nonumber -\frac{|M|  H(x^*)}{16\pi^3 }  \langle (1-|s'|^2)^{-1/2}, R_{\log,a}(1-|s'|^2)^{-1/2}\rangle\\
		&\nonumber +\frac{|M|  (\lambda_1 - \lambda_2)}{64 \pi^3} \langle (1-|s'|^2)^{-1/2}, R_{\infty,a}(1-|s'|^2)^{-1/2}\rangle\\
		&\nonumber + O(\epsilon\log\epsilon),
		\end{align}
	where $F$ is the unique solution to \eqref{F eq} and $R(x^*, x^*)$ is the evaluation at $(x,y) = (x^*, x^*)$ of the kernel $R(x,y)$ in \eqref{G expansion}. 
\end{prop}
Observe that in the case of the disc (i.e. $a=1$) we have that
$$\langle (1-|s'|^2)^{-1/2}, R_{\log,a}(1-|s'|^2)^{-1/2}\rangle = \pi^2 \left({8} \log 2 -{6}\right) $$ by Lemma \ref{integral of Rlog} and 
$$\langle (1-|s'|^2)^{-1/2}, R_{\infty,a}(1-|s'|^2)^{-1/2}\rangle= 0$$ by Lemma \ref{integral of Rinfty}. Thus the formula \eqref{Cepsa} simplifies to 
\begin{eqnarray} 
\label{Ceps}
		C_\epsilon := C_{\epsilon,1} &=&\frac{|M| K_a}{4a\epsilon \pi^2} -\frac{1}{4\pi} H(x^*) |M| \log \epsilon  + a R(x^*, x^*) |M|- F(x^*)\\\nonumber &-&\frac{|M|  H(x^*)}{16\pi } \left(8 \log 2 -6\right) + O(\epsilon\log\epsilon),
		\end{eqnarray}
when $a = 1$.

\subsection{Proof of Theorems \ref{main theorem disk} and \ref{main theorem}}
We now prove Theorem \ref{main theorem}. Theorem \ref{main theorem disk} follows from the explicit expression for $C_\epsilon$ in \eqref{Ceps}.

By the result of Appendix A we have that $u_\epsilon = \mathbb{E}[\tau_{\Gamma_{\epsilon, a}} | X_0 = x]$ solves the mixed boundary value problem \eqref{mixed bvp} so using Proposition \ref{expansion of psi and Ceps}, \eqref{ueps}, and \eqref{volume form in epsilon}, the expansion for $ \mathbb{E}[\tau_{\Gamma_{\epsilon, a}} | X_0 = x]$ is given by
$$ u_\epsilon(x) = \mathbb{E}[\tau_{\Gamma_{\epsilon, a}} | X_0 = x] = F(x) + C_{\epsilon,a} -|M| G(x,x^*) + r_\epsilon (x)$$
 for each $x\in M\backslash \Gamma_{\epsilon,a}$. Here $F$ is the unique solution to \eqref{F eq} and the remainder $r_\epsilon$ is given by
\begin{eqnarray}
\label{r eps remainder}
r_\epsilon(x) = \int_{\partial M}( G(x,y) - G(x,x^*))\partial_{\nu} u_\epsilon (y) {\rm dvol}_h(y).
\end{eqnarray}
Let $K\subset \overline M$ be a compact subset of $\overline M$ which has positive distance from $x^*$ and consider $x\in K$. Writing out this integral in the $x^\epsilon(\cdot; x^*)$ coordinate system and use \eqref{def of psi}, \eqref{volume form in epsilon}, and the expression of $\psi_\epsilon$ derived in Proposition \ref{expansion of psi and Ceps} we get
\begin{eqnarray}
r_\epsilon(x) &=& \epsilon  \int_{\D}  \frac{-|M|}{2\pi (1- |s'|^2)^{1/2}} L(x,\epsilon s') (1+\epsilon^2 Q_\epsilon(s')) ds'\\\nonumber
&+& a\epsilon^3 \int_{\D}\psi'_\epsilon (s') L(x,\epsilon s') (1+ \epsilon^2 Q_\epsilon(s')) ds'
\end{eqnarray}
for some function $L(x,s')$ jointly smooth in $(x,s') \in K\times \D$. The second integral formally denotes the duality between $H^{1/2}(\D)^*$ and $H^{1/2}(\D)$. The estimate for $\psi'_\epsilon$ derived in Proposition \ref{expansion of psi and Ceps} now gives for any integer $k$ and any compact set $K$ not containing $x^*$, $\| r_\epsilon\|_{C^k(K)}\leq C_{k,K}\epsilon$.

Our pseudodifferential characterization of $G_{\partial M}$ also allows us to compute the asymptotic of the average $\int_M u_\epsilon$. Indeed, integrating \eqref{ueps} over $M$ we get
\begin{equation}
\label{int of u}
\int_M u_\epsilon {\rm dvol}_g= \int_M F(x){\rm dvol}_g + C_{\epsilon,a} |M| + \int_M \int_{\Gamma_{\epsilon, a}} G(x,y) \partial_\nu u_\epsilon(y) {\rm dvol}_h(y) {\rm dvol}_g(x).
\end{equation}
We compute the last integral by noting that
$$v(x) := \int_{\Gamma_{\epsilon, a}} G(x,y) \partial_\nu u_\epsilon(y) {\rm dvol}_h(y)$$
is the unique solution to the Dirichlet boundary value problem:
\begin{eqnarray}
\label{v dirichlet}
\Delta_g v = 0,\ \ v(x) \mid_{\partial M} = \int_{\Gamma_{\epsilon, a}} G_{\partial M}(x,y) u_\epsilon(y) {\rm dvol}_h(y) \in H^{1/2}(\partial M).
\end{eqnarray}
We concluded the boundary value is in $H^{1/2}$ because $G_{\partial M}\in \Psi^{-1}_{cl}(\partial M)$ by \eqref{G expansion}.

Let a sequence of smooth functions $f_j \to \partial_\nu u_\epsilon$ in $H^{-1/2}(\partial M)$ and let $v_j$ solve 
$$\Delta_g v_j = 0,\ \ v_j(x) \mid_{\partial M} = \int_{\partial M} G_{\partial M}(x,y) f_j(y) {\rm dvol}_h(y).$$
Standard elliptic theory shows that $v_j \to v$ in $H^1(M)$. Therefore 
\begin{eqnarray*}
	\int_M \int_{\Gamma_{\epsilon, a}} G(x,y) \partial_\nu u_\epsilon(y)  &=& \lim_j  \int_M \int_{\partial M} G(x,y) f_j(y)\\ &=&\lim_j \int_{\partial M}  f_j(y) \int_{M} G(x,y) \\
	&=& \lim_j \int_{\partial M}  f_j(y) F(y)  \\
	&=& \langle \partial_\nu u_\epsilon , F\rangle = -F(x^*)|M| + O(\epsilon) ,
\end{eqnarray*}
where $F$ is the solution to the boundary value problem \eqref{F eq} and $\langle\cdot,\cdot\rangle$ denotes the pairing between $H^{-1/2}(\partial M)$ and $H^{1/2}(\partial M)$. The last equality comes from \eqref{compatible}, smoothness of $F$, and ${\rm supp}(\partial_\nu u_\epsilon) \subset \Gamma_{\epsilon, a}$. Inserting this into \eqref{int of u} we have
$$\int_M u_\epsilon {\rm dvol}_g = \int_M F(x) {\rm dvol}_g+ C_{\epsilon,a} |M| - F(x^*) |M| + O(\epsilon).$$
The constant $C_{\epsilon,a}$ is given by Proposition \ref{expansion of psi and Ceps}.

\end{section}

\begin{section}{Appendix A -Elliptic Equation for the first passage time}
\label{appendix A}
In this appendix we show that $u(x) := \mathbb{E}[\tau_{\Gamma} | X_0 = x]$ satisfies the boundary value problem \eqref{mixed bvp}. This is standard material but we could not find a suitable reference which precisely addresses our setting. As such we are including this appendix for the convenience of the reader.

Let $(M,g,\partial M)$ be an orientable compact connected Riemannian manifold with non-empty smooth boundary oriented by ${\rm dvol}_g$. Let also $(X_t, \mathbb{P}_x)$ be the Brownian motion on $M$ starting at $x$, that is, the stochastic process generated by the Laplace-Beltrami operator $\Delta_g$. Let $\Gamma$ be a geodesic ball on $\partial M$ with radius $\varepsilon > 0$. We denote by $\tau_{\Gamma}$ the first time the Brownian motion $X_t$ hits $\Gamma$, that is
	\begin{equation*}
	\tau_{\Gamma} := \inf \{ t\geq 0: X_t \in \Gamma\}.
	\end{equation*}
	We set
	\begin{equation*}
	\mathcal{P}_{\Gamma}(t,x) := \mathbb{P}[\tau_{\Gamma} \leq t| X_0 = x].
	\end{equation*}
	Let us note that $\mathcal{P}_{\Gamma}(t,x)$ is the probability that the Brownian motion hits $\Gamma$ before or at time $t$, and therefore, satisfies
	\begin{equation}\label{initial_condition}
	\mathcal{P}_{\Gamma}(0,x) = 0, \quad x\in M\setminus \Gamma,
	\end{equation}
	\begin{equation}\label{on_window}
	\mathcal{P}_{\Gamma}(t,x) = 1, \quad (t,x)\in  [0,\infty) \times \Gamma.
	\end{equation}
	
	Note that, for any compact subset $\Gamma\subset M$, it follows\footnote{Note that in \cite{GrigoryanSaloff-Coste} and \cite{ilmavirta}, the authors consider the manifold together with its boundary, and $C_c^{\infty}(M)$, $C_0^{\infty}(M)$ denote the set of smooth (up to the boundary) functions with compact support. In case of compact manifold, these sets coincide with $C^{\infty}(\overline{M})$.}
	\begin{equation*}
	\text{Cap}(\Gamma, M) := \inf_{u\in C^{\infty}(\overline{M}), \left.u\right|_{\Gamma} = 1} \int_{M} \|\nabla u\|^2 {\rm dvol}_g =0.
	\end{equation*}
	Then, \cite[Theorem 1.5]{ImperaPigolaSetti} implies that $(M,g,\partial M)$ is parabolic, that is, the probability that the Brownian motion ever hits any compact set $F$ with non-empty interior is $1$. Since $\Gamma \subset \partial M$ is connected with non-empty interior on $\partial M$, we can extend $M$ to a compact connected Riemannian manifold $\tilde{M}$ such that $\overline{\tilde{M} \setminus M}$ is compact with non-empty interior and $\overline{\tilde{M} \setminus M} \cap M = \Gamma$. Note that, the Brownian motion, starting at any point $M \setminus \Gamma$, hits $\overline{\tilde{M} \setminus M}$ if and only if it hits $\Gamma$. Therefore, the parabolicity condition of $(M, g)$ gives 
	\begin{equation}\label{lim_inf}
	\lim_{t \rightarrow \infty} \mathcal{P}_{\Gamma}(t,x) = 1, \quad x \in M.
	\end{equation}
	Further, let us define the mean first arrival time $u$, as
	\begin{equation}
\label{u as expected val}
	u(x) := \mathbb{E}[\tau_{\Gamma} | X_0 = x] := \int_{0}^{\infty} t d \mathcal{P}_{\Gamma}(t,x),
	\end{equation}
	where the integral is a Riemann-Stieltjes integral. To investigate $u$, let us recall some properties of $\mathcal{P}_{\Gamma}$. By Remmark 2.1 in \cite{GrigoryanSaloff-Coste}, it follows that
	\begin{equation*}
	1 - \mathcal{P}_{\Gamma}(t,x) = \left(e^{t\Delta_{mix}}1\right)(x),
	\end{equation*}
	where $e^{t\Delta_{mix}}$ is the semigroup with infinitesimal generator $\Delta_{mix}$, and $\Delta_{mix}$ is the Laplace operator $\Delta_g$ corresponding to the Dirichlet boundary condition on $\Gamma$ and Neumann boundary condition on $\partial M \setminus \Gamma$, which is defined as follows
	\begin{align}
	\label{dom of deltamix}
	&\mathrm{D}(\Delta_{mix}):=\{u\in H^1(M): \; \Delta_gu \in L^2 \;\left.u\right|_{\Gamma} = 0, \; \left.\partial_{\nu} u\right|_{\Gamma^c} = 0\}\\
	&\Delta_{mix} u = \Delta_g u \quad u \in \mathrm{D}(\Delta_{mix}).
	\end{align}
	In \eqref{dom of deltamix} we define $\partial_\nu u \in H^{-1/2}(\partial M)$ using the same method for defining the Dirichlet to Neumann map. That is, for $u\in H^1(M)$ such that $\Delta_g u \in L^2(M)$, the distribution $\partial_{\nu} u\left. \right|_{\partial M} \in H^{-1/2}(\partial M)$ acts on $f\in H^{1/2}(\partial M)$ via
	\begin{equation*}
	\langle \partial_{\nu}u\left. \right|_{\partial M} , f\rangle := \int_{M} \Delta u_g \overline{v_f} \; dvol_g + \int_{M} g(du,dv_f) \; dvol_g, 
	\end{equation*}
where $v_f\in H^1(M)$ is the harmonic extension of $f$.	We say that $\partial_{\nu} u\left. \right|_{\overline{\omega}} = 0$, for non-empty open set $\omega\subset \partial M$, if $\langle \partial_{\nu}u\left. \right|_{\partial M} , f\left. \right|_{\partial M}\rangle = 0$ for all $f \in H^{1/2}(\partial M)$ such that $f|_{\partial M\setminus \overline{\omega}} = 0$.

	Note that if $u$ sufficiently regular, for instance $u\in H^2(M)$, then $\langle\partial_{\nu}u\left. \right|_{\partial M} , f\rangle$ is equal to the boundary integral of $\partial_{\nu}u\left. \right|_{\partial M} $ and $ f$.
	
    In fact, $\Delta_{mix}$ can be equivalently defined by quadratic form; see Proposition \ref{Noperators} in Appendix B. Moreover, $\Delta_{mix}$ is the non-positive self-adjoint operator with the discrete spectrum, consisting of negative eigenvalues accumulating at $-\infty$; see Proposition \ref{Noperators} in Appendix.   Hence, $\Delta_{mix}$ satisfies the quadratic estimate
	\begin{equation*}
	\int_0^{\infty} \| t\Delta_{mix}(1 + t^2\Delta_{mix}^2)^{-1} u \|^2_{L^2}\frac{dt}{t} \leq C\|u\|^2,
	\end{equation*}
	for some $C>0$ and all $u\in L^2(M)$; see for instance \cite[p. 221]{McIntosh}. Therefore, $\Delta_{mix}$ admits the functional calculus defined in \cite{Nursultanov}. 
	\begin{Rem}
		The functional calculus in \cite{Nursultanov} is defined for a concrete operator, which is denoted by $T$ in the notation used in that article. However, $\Delta_{mix}$ satisfy all necessary conditions to admit this functional calculus.
	\end{Rem}
	Therefore, the semigroup $e^{t\Delta_{mix}}$, which is contracting by Hille-Yosida theorem \cite[Theorem 8.2.3]{Jost}, can be defined as follows
	\begin{equation*}
	e^{t\Delta_{mix}} u = \frac{1}{2\pi i}\int_{\gamma_{a,\alpha}} e^{t\zeta}(\zeta - \Delta_{mix})^{-1} u d\zeta, \qquad u\in L^2(M),
	\end{equation*}
	where $a\in (\tau , 0)$, $\alpha \in (0, \frac{\pi}{2})$, and $\gamma_{a,\alpha}$ is the anti-clockwise oriented curve:
	\begin{equation*}
	\gamma_{a,\alpha}:= \{ \zeta \in \mathbb{C}: \textrm{Re}\zeta \leq a, \text{ } |\textrm{Im}\zeta| = |\textrm{Re}\zeta - a|\tan\alpha\}.
	\end{equation*}
	Let $\varepsilon>0$ such that $a + \varepsilon < 0$. Then $\Delta_{mix} + \varepsilon$ is also a negative self-adjoint operator, and hence generates contracting semigroup, $e^{t(\Delta_{mix} + \varepsilon)}$, as above.
	
	By definition, we obtain, for $u \in L^2(M)$,
	\begin{align}\label{Delta_eps}
	e^{t\Delta_{mix}} u &= \frac{1}{2\pi i}\int_{\gamma_{a,\alpha}} e^{t\zeta}(\zeta - \Delta_{mix})^{-1} u d\zeta \\
	&\nonumber = \frac{e^{-t\varepsilon}}{2\pi i} \int_{\gamma_{a,\alpha}} e^{t(\zeta + \varepsilon)}(\zeta + \varepsilon - (\Delta_{mix} + \varepsilon))^{-1} u d\zeta\\
	&\nonumber = \frac{e^{-t\varepsilon}}{2\pi i} \int_{\gamma_{a + \varepsilon,\alpha}} e^{t\xi}(\xi - (\Delta_{mix} + \varepsilon))^{-1} u d\xi =  e^{-t\varepsilon} e^{t(\Delta_{mix} + \varepsilon)}u,
	\end{align}
	where $\gamma_{a + \varepsilon,\alpha} = \gamma_{a,\alpha} + \varepsilon \subset \{\textrm{Re}\xi < 0\}$.
	Let $f_1$ the constant function on $M$ equals $1$. By Theorem 8.2.2 in \cite{Jost}, we know, for $\lambda > 0$,
	\begin{equation*}
	(\lambda - (\Delta_{mix} + \varepsilon))^{-1}f_1 = \int_{0}^{\infty} e^{-\lambda t} e^{t(\Delta_{mix} + \varepsilon)}f_1 dt.
	\end{equation*}
	Let us choose $\lambda = \varepsilon$, then, by using \eqref{Delta_eps}, we obtain
	\begin{equation*}
	-\Delta_{mix}^{-1} f_1 = \int_{0}^{\infty} e^{-\varepsilon t} e^{t(\Delta_{mix} + \varepsilon)}f_1 dt = \int_{0}^{\infty} e^{t\Delta_{mix}} f_1 dt
	\end{equation*}
	and hence, 
	\begin{equation}\label{laplcae_P}
	\int_{0}^{\infty} 1 - \mathcal{P}_{\Gamma}(t,x) dt = -(\Delta_{mix}^{-1} f_1)(x) < \infty.
	\end{equation}
	Therefore, the dominated convergence theorem implies  
	\begin{equation*}
	\lim_{b \rightarrow \infty} \int_{0}^{b} \left(\mathcal{P}_{\Gamma}(b,x) - \mathcal{P}_{\Gamma}(t,x) \right)dt = \int_{0}^{\infty} 1 - \mathcal{P}_{\Gamma}(t,x) dt <\infty.
	\end{equation*}
	Hence, by using \eqref{u as expected val} and integration by parts, we obtain
	\begin{align*}
	u(x) &= \lim_{b \rightarrow \infty} \left( \mathcal{P}_{\Gamma}(b,x) b -\int_{0}^{b} \mathcal{P}_{\Gamma}(t,x) dt\right) = \lim_{b \rightarrow \infty} \int_{0}^{b} \left(\mathcal{P}_{\Gamma}(b,x) - \mathcal{P}_{\Gamma}(t,x) \right)dt < \infty\\
	& = \int_{0}^{\infty} 1 - \mathcal{P}_{\Gamma}(t,x) dt.
	\end{align*}
	Therefore, by \eqref{laplcae_P}, we obtain
	\begin{equation*}
	\Delta_{mix} u = -f_1=-1.
	\end{equation*}
	In particular, $u\in \mathrm{D}(\Delta_{mix})$, and hence,
	\begin{equation*}
	u \mid_{\Gamma} = 0, \qquad \partial_\nu u \mid_{\partial M\backslash \Gamma} = 0.
	\end{equation*}
We see that \eqref{mixed bvp} is satisfied.

\end{section}

\begin{section}{Appendix B - Quadratic Form}
Let $(M,g,\partial M)$ be a compact connected Riemannian manifold with non-empty smooth boundary. Let $\Gamma$ be a closed subset of $\partial M$ such that $\partial M\setminus \overline{\Gamma^c}$ is a non-empty open set. Consider the quadratic form
\begin{align}
a[u,v] := \int_{M} g(d u,d \overline{v}) {\rm dvol}_g, \qquad  u, v\in \mathrm{D}(a):=\{u\in H^1(M): \; \left.u\right|_{\Gamma} = 0 \}.
\end{align}
Note that $\mathrm{D}(a)$ is closed subspace of $H^1(M)$ containing $H_0^1(M)$ and $a[\cdot,\cdot]$ is a non-negative, closed, densely defined form. Therefore, by Friedrichs Theorem 2.23 in \cite{Kato}, it generates a non-negative self-adjoint operator $-\Delta_a$ in $L^2(M)$ whose domain is contained in $D(a)$ such that $(-\Delta_au,u)_{L^2(M)}=a[u,u]$ for $u\in \mathrm{D}(-\Delta_a)$.

Let us show that the resolvents of $-\Delta_a$ are compact. Assume that $s$ belongs to the resolvent set of $-\Delta_a$. Since $-\Delta_a(-\Delta_a - s)^{-1} : L^2(M) \rightarrow L^{2}(M)$ is bounded, it is suffices to show that $D(-\Delta_a)$, endowed with the graph norm, compactly embedded into $L^2(M)$. Since, for $u\in D(-\Delta_a)$, 
\begin{equation*}
\|du\|^2_{L^2(M)} = (du, du)_{L^2(M)}= a[u,u] = (-\Delta_au, u)_{L^2(M)}\leq \frac{1}{4}(\|-\Delta_au\|_{L^2(M)} + \|u\|_{L^2(M)})^2,
\end{equation*}
we see that any bounded sequence in $D(-\Delta_a)$, endowed with the graph norm, is bounded in $H^1(M)$, and hence, it contains a Cauchy subsequence in $L^2(M)$ by Rellich-Kondrachov theorem. This implies that the resolvents of $-\Delta_a$ are compact, and hence, the spectrum of $-\Delta_a$ is discrete, consisting of non-negative eigenvalues accumulating at $+\infty$. Assume that $\lambda = 0$ is an eigenvalue, and let $u_0 \in \mathrm{D}(a)$ be a corresponding eigenfunction. 

Then the Poincar\'e-Wirtinger inequality gives, for some $C>0$,
\begin{equation*}
\left\| u - \frac{1}{|M|} \int_{M} u {\rm dvol}_g\right\|_{L^2} \leq C\|d u\|_{L^2} = C(-\Delta_a u, u)_{L^2} = 0,
\end{equation*}
so that $u_0$ is a constant in $L_2(M)$. Since $u_0\in H^1(M)$, we conclude that $u_0=const$ in $L^2(\partial M)$, and hence, $u_0=0$ by choice of $\Gamma$. Therefrore $\lambda = 0$ is not an eigenvalue, and consequently, the spectrum of $-\Delta_a$ is positive. For sake of completeness, we prove the following well known result.

\begin{prop}\label{Noperators}
	Let $-\Delta_a$ be the operator defined above and $-\Delta_{mix}$ be the operator defined in Section \ref{appendix A}, then $-\Delta_a = -\Delta_{mix}$. In particular, $-\Delta_{mix}$ is a self-adjoint operator with the positive discrete spectrum accumulating at infinity.
\end{prop}
\begin{proof}
	Assume that $u$, $v\in H^1(M)$ and $\Delta_g u$, $\Delta_g u \in L^2(M)$. Let $V$ be the harmonic extension of $v\left.\right|_{\partial M}$, then $\omega:= V - v \in H^1_0(M)$ and $\Delta_g \omega\in L^2(M)$. By $H^2$-regularity, $\omega\in H^2(M)$. The generalized Green's identity gives
	\begin{equation*}
	-\int_{M}u \Delta_g \omega \; dvol_g + \int_{M} \Delta_g u \omega \; dvol_g = \langle\partial_{\nu} u\left.\right|_{\partial M}, \omega\left.\right|_{\partial M}\rangle - \langle u\left.\right|_{\partial M}, \partial_{\nu}\omega\left.\right|_{\partial M}\rangle,
	\end{equation*}
	where the first term of the right hand side vanishes since $\omega\in H_0^1(M)$. Therefore we get
	\begin{align*}
	0 &= -\int_{M}u \Delta_g \omega \; dvol_g + \langle u\left.\right|_{\partial M}, \partial_{\nu}\omega\left.\right|_{\partial M}\rangle + \int_{M} \Delta_g u \omega \; dvol_g \\
	&= \int_{M} g(d u ,d \overline{\omega}) \; dvol_g + \int_{M} \Delta_g u \omega \; dvol_g.
	\end{align*}
	Hence, we obtain
	\begin{equation*}
	\langle\partial_{\nu}u\left. \right|_{\partial M} , v\left. \right|_{\partial M}\rangle = \int_{M} \Delta_g u  \overline{v} \; dvol_g + \int_{M} g(d u ,d \overline{v}) \; dvol_g
	\end{equation*}
	for $u$, $v\in H^1(M)$ and $\Delta_g u$, $\Delta_g u \in L^2(M)$.

    Assume that $u\in \mathrm{D}(\Delta_{mix}) \subset \mathrm{D}(a)$ and $v\in\mathrm{D}(\Delta_a)$, then, by above formula,
	\begin{align*}
	a[u,v] = \int_{M} g(d u ,d \overline{v}) \; dvol_g = -\int_{M} \Delta_g u  \overline{v} \; dvol_g + \langle\partial_{\nu}u\left. \right|_{\partial M} , v\left. \right|_{\partial M}\rangle
	\end{align*}
	Note that the last term vanishes since $u\in \mathrm{D}(\Delta_{mix})$ and  $v\in \mathrm{D}(a)$, so that 
	\begin{equation*}
	a[u,v] = -\int_{M} \Delta_g u  \overline{v} \; dvol_g = -\int_{M} \Delta_{mix} u  \overline{v} \; dvol_g.
	\end{equation*}
	Since this holds for all $v\in \mathrm{D}(\Delta_a)$, it follows from Theorem 2.1., in \cite{Kato}, that $u\in \mathrm{D}(\Delta_a)$ and $\Delta_a u = \Delta_{mix} u$.
	
	Conversely, assume that $u \in \mathrm{D}(\Delta_a)$, then $-\Delta_g u =-\Delta_a u \in L^2(M)$. Then, it follows
    \begin{align*}
    \langle\partial_{\nu}u\left. \right|_{\partial M} , f\rangle =  \int_{M} \Delta u_g \overline{V_f} \; dvol_g + \int_{M} g(du,dV_f) \; dvol_g = a[u,V_f] - a[u,V_f] = 0.
    \end{align*}
    for any $f \in H^{1/2}(\partial M)$ such taht $f\left.\right|_{\Gamma} = 0$. This means that $\partial_{\nu}\left.u\right|_{\Gamma^c} = 0$, so that $u\in D(\Delta_{mix})$ and $\Delta_a u = \Delta_{mix} u$.

\end{proof}
\end{section}

\bibliographystyle{plain}
\bibliography{references}

\setlength{\parskip}{0pt}





\end{document}